\documentclass[11pt, a4paper]{article}
\usepackage{amsmath, amssymb, amsfonts, amsthm}
\usepackage[pdftex]{graphicx, color}

\newcommand{\E}{{\bf{E}}}
\newcommand{\PP}{{\bf{P}}}
\newcommand{\Var}{{\bf{Var}}}

\newtheorem{tm}{Theorem}
\newtheorem{lem}{Lemma}
\newtheorem{col}{Corollary}

\usepackage{hyperref}
\hypersetup{
    colorlinks=true,
    linkcolor=blue,
    filecolor=magenta,      
    urlcolor=cyan,
}

\begin{document}

\parindent=0pt

\smallskip
\par\vskip 3.5em
\centerline{\Large  On the strength of connectedness  of  unions of random graphs}

\vglue1truecm

\centerline {Mindaugas Bloznelis}

\bigskip

\centerline{Vilnius University, Faculty of Mathematics and Informatics}
\centerline{
\, \ Didlaukio 47, LT-08303 Vilnius, Lithuania} 

\vglue 1truecm

\qquad E-mail:\, \
mindaugas.bloznelis@mif.vu.lt

\vglue1truecm

\abstract{
Let $G_1,\dots, G_m$ be independent
identically distributed
random subgraphs of the complete graph ${\cal K}_n$. 
We 
analyse the threshold behaviour of the strength of connectedness  of the union 
$\cup_{i=1}^mG_i$ defined on the vertex 
set of ${\cal K}_n$.  
Let $a=\min\{t\ge 1:\, {\bf P}\{\delta(G_1)=t\}>0\}$ be 
the minimal non zero vertex degree  attained with positive 
probability. Given $k\ge 0$ let $\lambda(k)=\ln n+k\ln\frac{m}{n}-\frac{m}{n}
{\bf E} X$,  where
$X$ stands for the number of non
isolated vertices of $G_1$.
Letting $n,m\to+\infty$ 
we show that 
${\bf P}\{\cup_{i=1}^mG_i$ is $a(k+1)$-connected$\}
\to
1
$ for
$\lambda(k)\to -\infty$, and
${\bf P}\{\cup_{i=1}^mG_i$ is $ak+1$-connected$\}
\to
0
$ for
$\lambda(k)\to +\infty$.
In particular, the 
connectivity strength of the union graph 
$\cup_{i=1}^mG_i$ increases in steps of size $a$.
Our results are obtained in a more general setting where the contributing random subgraphs do not need to be identically distributed.} 

\par\vskip 2em

\section{Introduction and results}

After the seminal work of Erd\H os and R\'enyi \cite{ErdosRenyi1959,ErdosRenyi1961}  the strength of connectedness of large random graphs has attracted considerable attention in the literature and  remains an area of active research, see, for example,
\cite{BergmanLeskela},
\cite{BlackburnGerke2009}, 
\cite{Devroye_Fraiman_2014},
\cite{GodehardJaworskiRybarczyk2007},
\cite{Penrose_1999},
\cite{Rybarczyk2011}
\cite{Shang_2023},
\cite{Wormald_1981}, 
 \cite{YaganMakowski2012},
 \cite{Zhao_yagan_Gligor_2017} for a variety of random graph models considered.

In the present paper we study the strength of connectedness
of random graph unions.
Let $G_1=({\cal V}_1,{\cal E}_1),\dots, G_m=({\cal V}_m,
{\cal E}_m)$ be random subgraphs of the complete graph 
${\cal K}_n$.  
We denote ${\cal V}=\{1,\dots, n\}=:[n]$
 the vertex set of ${\cal K}_n$ and consider
   the union graph $G_{[n,m]}=G_1\cup\cdots\cup G_m$ 
 on the vertex set ${\cal V}$ and with the edge 
 set ${\cal E}_1\cup\cdots\cup{\cal E}_m$.
We impose the following  conditions on the sequence 
$G_1,\dots, G_m$: (i) subgraphs $G_1,\dots, G_m$ are selected  independently at random;
 (ii)  given $|{\cal V}_i|$, the vertex set ${\cal V}_i$
is distributed uniformly across the class of subsets of ${\cal V}$ of size $|{\cal V}_i|$;
(iii) given ${\cal V}_i$ the distribution of $G_i$ is 
invariant under permutations of  vertices 
of ${\cal V}_i$. 
Note that we do not specify the edge distributions of 
$G_i$, which may vary for $i=1,\dots, m$. 

Let us consider  two examples.

{\bf{Example 1}}. Let  $F_1,\dots, F_m$
be a sequence of (non-random) graphs without isolated vertices. 
Let ${\cal V}_1',\dots,{\cal V}'_m$ denote their vertex sets.
We
 map each ${\cal V}'_i$ in ${\cal V}$ by 
an injective map $\pi_i$, say. In this way we obtain
a copy of $F_i$ with the vertex 
set ${\cal V}_i=\pi_i({\cal V}'_i)\subset {\cal V}$.
Assuming that each 
injection $\pi_i$  is selected uniformly at random 
(from the class of injections ${\cal V}'_i\to{\cal V}$)
and 
independently across $i=1,\dots, m$ we obtain a sequence 
of random copies of $F_1,\dots, F_m$, which we denote $G_1,\dots, G_m$. 
In the particular case, where $F_1=\cdots=F_m={\cal K}_2$,
the random graph $G_{[n,m]}$ is a union of $m$ randomly inserted edges, which may overlap.

{\bf{Example 2}}. 
Let 
$(Y_1,Q_1),\dots,(Y_m,Q_m)$ be 
independent bivariate random variables 
taking values in $[n]\times [0,1]$. 
Given $(Y_i,Q_i)$, we generate Bernoulli random graph  on $Y_i$ vertices and with edge density $Q_i$. The resulting random graph is denoted $G'_i$.
We assume that random graphs $G'_1,\dots, G'_m$ are independent. 
We  map
vertex sets of $G'_1,\dots, G'_m$ in ${\cal V}$ by 
random injective maps as in Example 1 
above. In this way we obtain copies 
of $G'_1,\dots, G'_m$, which we denote
$G_1,\dots, G_m$. In the particular case, where $\PP\{Q_i=1\}=1$ for $1\le i\le m$, the random graph $G_{[n,m]}$ is the union of  cliques of sizes $Y_1,\dots, Y_m$ randomly scattered over the vertex set ${\cal V}$.
 
In Theorem \ref{T1} below we show the connectivity 
threshold for the union graph $G_{[n,m]}$. We find that 
the only parameter that defines the threshold is the 
average number of non isolated vertices
in $G_1,\dots, G_m$. 
In Theorem \ref{T2} we show the $k$-connectivity threshold for $G_{[n,m]}$. We notice that the connectivity strength of
$G_{[n,m]}$ increases in steps of size $a$, where 
$a\ge 1$ is the minimal non-zero vertex degree
attained with positive probability by a fraction of subgraphs $G_1,\dots, G_m$, for detailed definition see (\ref{a}) below.

Before presenting our results, we introduce some notation.
Given two sequences of positive 
numbers $\{y_n\}$ and $\{z_n\}$ we write
$y_n\prec z_n$ whenever $z_n-y_n\to+\infty$ as $n\to+\infty$.
By $\delta(G)$ we denote the minimal degree of a graph $G$.
For $v\in{\cal V}_i$ we denote by $d_i(v)$ the number of vertices of ${\cal V}_i$ linked to $v$ by the edges of $G_i$. We put $d_i(u)=0$ for $u\notin {\cal V}_i$.
By $X_i(t)$ we denote the number of vertices $
v\in{\cal V}_i$ with $d_i(v)=t$. 
By $X_i=X_i(1)+X_i(2)+\cdots$ we denote the number of non isolated vertices of $G_i$. 
Let $i_*$ be a number selected  uniformly at random from
$[m]$ and independently of $G_{1},\dots, G_{m}$. 
The random variable  $X_{i_*}$ represents a mixture of $X_1,\dots, X_m$ with the probability distribution
$\PP\{X_{i_*}=s\}=\frac{1}{m}\sum_{i=1}^m\PP\{X_{i}=s\}$, $s=0,1,\dots$. Likewise the random variable $X_{i_*}(t)$ is a mixture of $X_1(t),\dots, X_m(t)$.
We denote
\begin{align*}
&
\alpha
=
\PP\{X_{i_*}>0\}
=
\frac{1}{m}\sum_{i=1}^n\PP\{X_{i}>0\},
\qquad
\kappa=\E X_{i_*}
=\frac{1}{m}\sum_{i\in[m]}\E X_{i},
\\
&
\kappa(t)=\E X_{i_*}(t)
=\frac{1}{m}\sum_{i\in[m]}\E X_{i}(t),
\qquad
\lambda(k)=\ln n+k\ln\frac{m}{n}-\frac{m}{n}\kappa.
\end{align*}

\noindent In Theorems \ref{T1}  and \ref{T2}  we consider a sequence of random graphs $G_{[n,m]}$, where $m=m(n)\to+\infty$ as $n\to+\infty$. To indicate the dependence on $n$ we 
write
$G_{[n,m]}=G_{n,1}\cup\cdots\cup G_{n,m}$, 
where $G_{n,1}=({\cal V}_{n,1},{\cal E}_{n,1})$, $\dots$, 
$G_{n,m}=({\cal V}_{n,m},{\cal E}_{n,m})$ are random 
subgraphs of ${\cal K}_n$. We write $X_{n,i}$ 
(respectively $X_{n,i}(t)$)  for the number of non 
isolated (respectively degree $t$) vertices of $G_{n,i}$, 
$1\le i\le m$, and  define $X_{n,i_*}$  and 
$X_{n,i_*}(t)$ in a natural way. Furthermore, we write 
$\alpha_n$, $\kappa_n$, $\kappa_n(t)$ and $\lambda_n(k)$. We drop the subscript $n$ wherever this does not cause  ambiquity.

\begin{tm}\label{T1}
Let $c$ be a real number.
Let $n\to+\infty$. 
Assume that $m/n\to +\infty$ and $m=O(n\ln n)$. 
Assume that  the sequence of random variables $\{X_{n,i_*}\ln(1+X_{n,i_*}), \ n\ge 1\}$ is uniformly integrable, that is, 
\begin{align}
\label{2025-11-24}
\lim_{t\to+\infty}
\sup_{n}
\E\left(X_{n,i_*}\ln(1+X_{n,i_*})
{\mathbb I}_{\{X_{n,i_*}>t\}}\right)
=
0.
\end{align}
Here ${\mathbb I}_{\{X_{n,i_*}>t\}}$ denotes the indicator 
of the event $\{X_{n,i_*}>t\}$. Assume that 
\begin{align}
\label{2025-12-02}
\liminf_n\alpha_n>0.
\end{align}
Then 
\begin{equation}
\label{T1+}
\lim_{n\to+\infty}
\PP\{G_{[n,m]}{\rm{\ is \ connected \ }}\}
=
\begin{cases}
1,
\qquad 
\
{\rm{for}}
\quad \lambda_{n}(0)\to-\infty;
\\
e^{-e^c},
\quad 
{\rm{for}}
\quad \lambda_{n}(0)\to c;
\\
0,
\qquad 
\
{\rm{for}}
\quad \lambda_{n}(0)\to+\infty.
\end{cases}
\end{equation}
\end{tm}
\medskip
\noindent

We note that condition (\ref{2025-11-24}) 
is very mild and generally 
cannot be relaxed. This is demonstrated by an example 
in \cite{Daumilas_Mindaugas2023}.
We mention that
conditions (\ref{2025-11-24}), (\ref{2025-12-02}) imply  
\begin{align}
\label{2025-12-02+2}
0<\liminf_n\kappa_n\le \limsup_n\kappa_n<\infty,
\end{align}
where the first inequality  
follows 
from (\ref{2025-12-02})  combined with the simple 
inequality  $\kappa_n\ge 2\alpha_n$.

{\bf Example 1} (continued). Let $F_1,F_2,\dots$ be a sequence of (non-random and non-empty) graphs without isolated 
vertices. For $i=1,2,\dots$, let $x_i$  denote the number 
of vertices of $F_i$. 
Let $m=m(n)$  satisfies  $m(n)\to+\infty$  and  $m(n)=O(n\ln n)$ as $n\to+\infty$. Assume that   $\max_{i\in [m]}x_i\le n$ for each $n$, and 
$\sup_m\frac{1}{m}\sum_{i\le m: \ x_i>t}x_i\ln x_i\to 0$ as $t\to+\infty$. 
 Theorem \ref{T1} tells us that  (\ref{T1+}) holds with $\lambda_{n}(0)=\ln n-\frac{1}{n}\sum_{i=1}^mx_i$.

 {\bf Example 2} (continued). Assume that $(Y_1,Q_1), (Y_2,Q_2),\dots $ are independent and  iden\-tically distributed (iid) random variables taking values in 
 $\{0,1,2,\dots\}\times[0,1]$. For each $n$ and $1\le i\le m$ 
 let $G'_{n,i}$ be Bernoulli random graph with 
 ${\hat Y}_{n,i}:=\min\{Y_i,n\}$ vertices and with the edge density $Q_i$. 
 Assuming that 
 $\PP\{Y_1\ge 2, Q_1>0\}>0$
 and 
 $\E (Y_1h(Y_1,Q_1)\ln(1+Y_1))<\infty$, 
 where $h(k,q)=1-(1-q)^{k-1}$,
we 
verify conditions (\ref{2025-11-24}), (\ref{2025-12-02}) of Theorem \ref{T1}. Moreover we show  that 
$\kappa_n=\kappa'+o(\ln^{-1}n)$, where $\kappa'=\E (Y_1h(Y_1,Q_1))$
 does not depend on $n$. Hence 
 (\ref{T1+}) holds with 
 $\lambda_{n}(0)$ replaced by $\lambda'_{n}=\ln n-\frac{m}{n}\kappa'$. Note that  $kh(k,q)$ is the 
expected number of non-isolated vertices in Bernoulli 
random graph $G(k,q)$.

 To formulate the $k$-connectivity threshold we need a characteristic representing joint 
 minimal positive degree of the contributing graphs $G_{n,1},\dots, G_{n,m}$. We introduce the following condition:
There exists integer $a\ge 1$  such that
\begin{align}
\label{a}
\liminf_n\E X_{n,i_*}(a)>0,
\qquad
\E X_{n,i_*}(t)=0 
\quad
{\text{for}}
\quad 
1\le t<a
\quad
{\text{and\ all}} \quad n.
\end{align}
Note that for $a=1$ the second part of condition (\ref{a}), namely $\E X_{n,i_*}(t)=0$, $1\le t<a$, is void. Condition 
(\ref{a}) means that a fraction of contributing graphs $G_{n,1},\dots, G_{n,m}$ shares the same minimal positive degree and none of $G_{n,i}$, $1\le i\le m$ has a smaller one.

 We recall that a graph is called $k$-vertex (edge) connected if the removal of any $k-1$ vertices (edges) does not disconnect
it.
We write $G\in{\cal C}_k$ whenever a finite graph $G$ is $k$ vertex connected.

\begin{tm}\label{T2}
  
Let  $a\ge 1$ and $k\ge 0$ be integers. Let $n\to+\infty$.
Assume that  $m/n\to+\infty$ and $m=O(n\ln n)$.
Assume that 
(\ref{2025-12-02}),
(\ref{a})
hold and 
\begin{align}
\label{moment_condition}
\limsup_n\E X_{n,i_*}^{(k+1)a+1}<\infty. 
\end{align}
Then 
\begin{align}
\label{2026-01-20+1}
&
\PP\{G_{[n,m]}\in {\cal C}_{ka+1}\}
\to
0
\quad
\
\,
{\text{ for}}
\quad
\lambda_n(k)\to+\infty,
\\
\label{2025-08-25+r}
&
\PP\{G_{[n,m]}\in {\cal C}_{(k+1)a}\}
\to
1
\quad
{\text{ for}}
\quad
\lambda_n(k)\to-\infty.
\end{align}
Furthermore, for 
\begin{align}
\label{2026-01-21}
\ln n+k\ln\frac{m}{n}\prec\kappa\frac{m}{n}\prec 
\ln n+(k+1)\ln\frac{m}{n}
\end{align}
we have  
\begin{align}
\label{degree}
\PP\{\delta(G_{[n,m]})=(k+1)a\}\to 1.
\end{align}
\end{tm}

Condition (\ref{a}) can be relaxed in the sense that the strict identity $\E X_{n,i_*}(t)=0$ on the right of (\ref{a}) can be replaced by a weaker condition that $\max_{1\le t<a}\E X_{n,i_*}(t)$ tends to $0$ sufficiently fast as $n\to+\infty$. It seems that the rate of 
$\ln^{-1} n$ would suffice.

 {\bf Example 2} (continued). Assume that $(Y_1,Q_1),(Y_2,Q_2),\dots$ are iid and $\PP\{Y_1\ge 2, Q_1>0\}>0$. Let us check condition (\ref{a}).
 Depending on the distribution of $(Y_1,Q_1)$ we distinguish two cases. For 
 $\PP\{Y_1\ge 3, Q_1\in(0,1)\}>0$ 
 random graph $G'_{n,1}$ may attain any configuration 
 of edges on 
 ${\hat Y}_{n,1}=\min\{Y_1,n\}$ vertices with positive probability. 
 Hence $\PP\{\delta(G'_{n,1})=1\}>0$. In this case  condition (\ref{a}) holds with 
$a=1$. For $\PP\{Y_1\ge 3, Q_1\in(0,1)\}=0$ the random graph $G'_{n,1}$ is either a clique 
or an independent set, both of size ${\hat Y}_{n,1}$. In this case condition (\ref{a}) holds with  $a=\min\{t\ge 1: \PP\{Y_1=t\}\}$.
We also observe that in view of inequalities
$X_{n,i}\le {\hat Y}_{n,i}\le Y_i$
 condition (\ref{moment_condition}) is met whenever $\E Y_1^{(k+1)a+1}<\infty$. 

\medskip

{\bf Related work}. Random graph $G_{[n,m]}$ generalises the 
classical Erd\H os-R\'enyi random graph $G(n,M)$ on  $n$ vertices and with
$M$ distinct edges selected uniformly at random.
For $G_1=\cdots=G_m={\cal K}_2$  graph $G_{[n,m]}$ represents an instance of  $G(n,M)$, where the   number of edges $M\le m$ is random, because some of $G_i$ may overlap. In the parametric range $m\asymp n\ln n$ the overalps are rare (we have  $\E M=m-O(\ln^2n)$) and the connectivity thresholds for both models $G(n,m)$ and $G_{[n,m]}$, with $G_i={\cal K}_2$ $\forall i$, are the same.

In network modelling literature  unions of random graphs $G_1\cup\cdots\cup G_m$
are used to model networks of overlapping communities $G_1,\dots, G_m$, see, e.g, 
 \cite{Hofstad2024_book_2} and the references therein.
 In the sparse parametric regime $m\asymp n$ they admit power law degree distributions and tunable clustering coefficients.
A union of independent  Bernoulli random graphs, where
configuration of the vertex sets ${\cal V}_1,\dots, {\cal V}_m$ is defined by a design that features non-negligible
overlaps
  have been used by \cite{Yang_Leskovec2014} 
as a
benchmark network model 
(called community affiliation graph)
for studying overlapping community detection algorithms.
The random graph of Example 2 represents a null model of the community affiliation graph.
We mention that  
for $m\asymp n$ the phase transition in the size of the largest component and percolation of 
the union of Bernoulli graphs has been shown in
\cite{Lasse_Mindaugas2019}; weak local limit has been studied  in \cite{Kurauskas_2022}. 
In the particular case where contributing  Bernoulli random graphs have unit edge densities 
($Q_i\equiv 1$ $\forall i$)
we obtain a union of randomly scattered cliques. Such a 
union  is called  the clique graph of  non-uniform random hypergraph  defined by the collection of hyperedges ${\cal V}_1,\dots, {\cal V}_m$. Another term   used in the literature  for a union of cliques is 'passive random intersection graph' indicating  
a 
connection to 
random intersection graphs, \cite{GodehardJaworski_2001},
\cite{BGJKR_2015_Properties}.
In the very special case, where clique sizes have the common Binomial distribution $Binom(n,p)$ the connectivity ($1$-connectivity) threshold has been shown in \cite{Singer_1995}. 
The connectivity threshold for unions of iid cliques with sizes having a general probability distribution 
has been established in 
\cite{GodehardJaworskiRybarczyk2007},
\cite{BergmanLeskela},
\cite{Daumilas_Mindaugas2023}. The latter paper \cite{Daumilas_Mindaugas2023}
 also addresses the unions of Bernoulli graph with arbitrary edge densities $Q_i$. Finally, the $k$-connectivity threshold for unions of Bernoulli random graphs 
has been shown in \cite{Daumilas_Mindaugas_Rimantas_2025} and \cite{{Bloznelis_clique graph_2025}}.

 In the remaining part of the paper we prove Theorems \ref{T1} and \ref{T2} and show how the  connectivity threshold for the union of Bernoulli random graph of Example 2 follows from Theorem \ref{T1}. Proofs are given in Section 2. Auxiliary results are collected in Section 3.

\section{Proofs}

\subsection{Notation}

Given two sequences of positive 
numbers $\{y_n\}$ and $\{z_n\}$ we write
$y_n\prec z_n$ whenever $z_n-y_n\to+\infty$ as $n\to+\infty$.
We write $y_n\asymp z_n$ whenever 
$0<\liminf_n\frac{y_n}{z_n}\le \limsup_n\frac{y_n}{z_n}<+\infty$.
For a sequence of random variables $\xi_1,\xi_2,\dots$ we write 
$\xi_n=o_P(1)$ whenever 
$\lim_{n\to+\infty}\PP\{|\xi_n|>\varepsilon\}=0$ for each $\varepsilon>0$.
Given a sequence of events ${\cal A}_n, n\ge 1$ we say 
that event ${\cal A}_n$ occurs with high probability 
if $\PP\{{\cal A}_n\}\to 1$ as $n\to +\infty$. 
By ${\mathbb I}_{\cal B}$ we denote the indicator 
function of event (or set) 
${\cal B}$.
For $u,v\in {\cal V}$ we denote
\begin{align*}
d'(v)
&
=
\sum_{i\in[m]}{\mathbb I}_{\{d_i(v)>0\}},
\qquad
d'(u,v)=\sum_{i\in[m]}{\mathbb I}_{\{d_i(v)>0,\,
d_i(u)>0\}},
\\
N_k
&
=
\sum_{v\in{\cal V}}{\mathbb I}_{\{d'(v)=k\}},
\qquad
N'_k
=
\sum_{v\in {\cal V}}\sum_{u\in{\cal V}\setminus\{v\}}
{\mathbb  I}_{\{d'(v)= k\}}
{\mathbb  I}_{\{d'(u,v)\ge 2\}},
\qquad
k\ge 0,
\\
d_*'(v)
&
=
\sum_{i\in[m]}
{\mathbb I}_{\{d_i(v)=a\}},
\qquad
N_{*k}
=
\sum_{v\in{\cal V}}
{\mathbb I}_{\{d_*'(v)=d'(v)=k\}},
\quad
k\ge 0.
\end{align*}
By $d(v)$ we denote the degree of  $v\in {\cal V}$ in 
$G_{[n,m]}$.
We will call $G_1,\dots, G_m$ communities. 
We remark that event $\{d_*'(v)=d'(v)=k\}$ means that 
$v$ is a non-isolated vertex of  $k$ communities and in each of these communities it has the 
(minimal possible) number of neighbours  $d_i(v)=a$.
Let ${\cal A}$ denote the event  that
$d'(u,v)\ge 3$ for some $u,v\in {\cal V}$.

Let ${\tilde{\cal V}}_i\subset {\cal V}_i$ denote the 
set of vertices $v\in{\cal V}_i$ with  $d_i(v)\ge 1$.
We call a collection of sets  
${\tilde {\cal V}}_{i_1},\dots, {\tilde{\cal V}}_{i_k}$  
$k$-{\it blossom} 
centered at $w$ if each pair of  sets has 
 the only common element $w$ (i.e., 
${\tilde {\cal V}}_{i_r}\cap{\tilde {\cal V}}_{i_\ell}=\{w\}$, for $r\not=\ell$). The sets 
${\tilde {\cal V}}_{i_1},\dots, {\tilde{\cal V}}_{i_k}$  are called petals of the blossom.
 Note that  event $N'_k=0$ 
implies that any vertex $u$ with $d'(u)=k$ is the central vertex of a $k$-{\it blossom}.

By $\eta_k$ we denote the number of connected components 
of $G_{[n,m]}$ 
of size $k$ ($=$ having $k$ vertices);
${\cal A}_i=
\left\{\sum_{i\le k\le n/2}\eta_k\ge 1\right\}$
denotes the event that $G_{[n,m]}$ has a component on $k$ 
vertices for 
some $i\le k\le n/2$.
Note that $\eta_1$ is the number of isolated vertices 
of $G_{[n,m]}$.
Given $S\subset {\cal V}$ we denote $G_{[n,m]}\setminus S$ the subgraph of $G_{[n,m]}$ induced by the vertex set ${\cal V}\setminus S$.
We introduce events 
 \begin{align*}
 {\cal P}_k
 &
 =\{G_{[n,m]}\setminus S   {\text{ has no isolated vertex for any }} S\subset {\cal V}, 
 |S|\le k \},
 \\
 {\cal B}_k
 &
 =\Bigl\{ \exists  S\subset {\cal V}: |S|\le k, G_{[n,m]}\setminus S \,
{\text{has  a  component  on}}
\
r
\
{\text{vertices  for  some}}
\
2\le r\le \frac{n-|S|}{2}
\Bigr\}.
\end{align*}

\subsection{Proof of Theorem \ref{T1}}

The proof of Theorem \ref{T1} follows the scheme presented in \cite{FriezeKaronski}.
We 
establish an expansion property and show the asymptotic Poisson distribution for the number of isolated vertices.
 A novel contribution is  an estimate of the probability of a link between a given set of vertices and its complement 
 shown in Lemma \ref{basic} below.
 We mention that the result of Theorem \ref{T1}
 is obtained under rather weak condition (\ref{2025-11-24})
 on the distribution of $X_{i_*}$. 
For this reason, the proof is somewhat technical.

The section is organized as follows. We first state Lemmas \ref{Lemma_1G} and \ref{Lemma_2}, which are the main ingredients of the proof of  Theorem \ref{T1}. Then we prove Theorem \ref{T1}. Afterwards we prove Lemmas \ref{Lemma_1G}, \ref{Lemma_2}. At the end of the section we show the
connectivity threshold for unions of Bernoulli 
graphs of Example 2.

\begin{lem}
\label{Lemma_1G} Let $c_1>0$ and $c_2$ be real numbers. 
Let $n,m\to+\infty$. Assume that $m\le c_1n\ln n$ and 
$\lambda_{n}(0)\le c_2$. Assume that (\ref{2025-11-24}),
(\ref{2025-12-02}) hold.   Assume that either 
$\lambda_{n}(0)\to c$ or $\lambda_{n}(0)\to-\infty$.
Then $\PP\{{\cal A}_2\}=o(1)$.
\end{lem}

\begin{lem}\label{Lemma_2} 
Let $c$ be a real number.
Let 
$m,n\to+\infty$.
Assume that $m=O(n\ln n)$. Assume that (\ref{2025-11-20})
holds. Then  
\begin{align}
\label{2025-11-20+6}
\PP\{\eta_1=0\}
\to
\begin{cases}
0,
\qquad
\quad
\
{\text{for}}
\qquad
\lambda_{n}(0)
\to +\infty;
\\
e^{-e^c},
\qquad
{\text{for}}
\qquad
\lambda_{n}(0)
\to c;
\\
1,
\qquad
\quad
\
{\text{for}}
\qquad
\lambda_{n}(0)
\to -\infty.
\end{cases}
\end{align}
Furthermore, for $\lambda_{n}(0)\to c$ 
the probability distribution of $\eta_1$ converges 
to  the Poisson distribution with mean value $e^{c}$.
\end{lem}

\begin{proof}[Proof of Theorem \ref{T1}]
For $\lambda_{n}(0)\to+\infty$ Lemma \ref{Lemma_2}
implies that $G_{[n,m]}$ contains an isolated vertex whp.
Hence $G_{[n,m]}$ is disconnected whp.

In the  remaining cases $\lambda_{n}(0)\to c$ and $\lambda_{n}(0)\to-\infty$ we write
 the probability that $G_{[n,m]}$ is disconnected in the form
\begin{displaymath}
\PP\{{\cal A}_1\}=\PP\{{\cal A}_2\cup\{\eta_1\ge 1\}\}
=\PP\{\eta_1\ge 1\}
+\PP\{{\cal A}_2\setminus \{\eta_1\ge 1\}\}.
\end{displaymath}
We show in Lemma 
\ref{Lemma_1G}  that 
$\PP\{{\cal A}_2\}=o(1)$. Hence 
$\PP\{{\cal A}_1\}
=\PP\{\eta_1\ge 1\}+o(1)$.
Now (\ref{T1+}) follows from Lemma \ref{Lemma_2} (note that condition (\ref{2025-11-20}) of Lemma \ref{Lemma_2}
follows from (\ref{2025-11-24})).
\end{proof}

Before the proof  of Lemmas \ref{Lemma_1G}, \ref{Lemma_2} we
collect auxiliary results. 
Let us show that (\ref{2025-11-24}) implies
\begin{align}
\label{2025-11-20}
\E X_{n,i_*}^2
\left(= \frac{1}{m}\sum_{i=1}^m\E X^2_{n,i}
\right)
=
o\left(\frac{n}{\ln n}\right).
\end{align}

We write, for short, $\xi_n=X_{n,i_*}\ln(1+X_{n,i_*})$. Note that the (\ref{2025-11-24}) implies  $\sup_n\E \xi_n<\infty$.
To show (\ref{2025-11-20}) we split the integral
\begin{align*}
\E X^2_{n,i_*}
&
=
\E 
\left(
X^2_{n,i_*}
\left({\mathbb I}_{\{X_{n,i_*}<\sqrt n\}}
+
{\mathbb I}_{\{X_{n,i_*}\ge\sqrt n\}} \right)
\right)
\\
&
\le 
\frac{\sqrt{n}}{1+\ln(1+\sqrt{n})}
\E \left(\xi_n{\mathbb I}_{\{X_{n,i_*}<\sqrt n\}}
\right)
 +
\frac{n}{1+\ln(1+n)}
\E \left(\xi_n{\mathbb I}_{\{X_{n,i_*}\ge\sqrt n\}}
\right).
\end{align*}
The first term of the sum is $O\left(\frac{\sqrt{n}}{\ln(\sqrt{n})}\right)$ since $\E \left(\xi_n{\mathbb I}_{\{X_{n,i_*}<\sqrt n\}}
\right)\le\E \xi_n$ is bounded.
The second term is $o\left(\frac{n}{\ln(n)}\right)$
since the sequence $\{\xi_n, n\ge 1\}$ is uniformly integrable.

It follows from 
 (\ref{2025-11-20}) that for some sequence 
 $\phi'_n\downarrow 0$  we have
$\E X_{n,i_*}^2\le \frac{n}{\ln n}\phi'_n$. For $m=O(n\ln n)$ we have, in addition, $\sum_{i\in[m]} \E X^2_{n,i}\le n^2\phi''_n$ for some sequence $\phi''_n\downarrow 0$.  
Letting $\phi_n=\max\{\phi'_n,\phi''_n\}$ we have
\begin{align}
\label{phi}
&
\E X_{n,i_*}^{b+1}
\le 
n^{b-1}
\E X_{n,i_*}^2
\le 
\frac{n^b}{\ln n}\phi_n,
\qquad
{\text{for}}\quad b=1,2,\dots,
\\
\label{phi+}
&
\max_{i\in [m]}
(\E X_{n,i})^2
\le
\max_{i\in [m]}
\E X^2_{n,i}
\le 
\sum_{i\in[m]} 
\E X^2_{n,i}
\le 
n^2
\phi_n.
\end{align}
In the first step of (\ref{phi+})  we used inequality 
$(\E X_{n,i})^2
\le \E X_{n,i}^2$, which follows by the Cauchy-Schwarz inequality.

 \begin{proof}[Proof of Lemma \ref{Lemma_1G}]
 For a subset $U\subset {\cal V}$ we denote by ${\cal B}_U$ the event that
$U$ induces a  connected component of 
$G_{[n,m]}$. 
We denote by ${\cal D}_U$  the event that
$G_{[n,m]}$ has no edges connecting vertex sets $U$ and 
${\cal V}\setminus U$. 
Note that 
$\PP\{{\cal B}_U\}\le \PP\{{\cal D}_U\}$
 and  
\begin{equation}
\label{Y+}
\eta_k=\sum_{U\subset V,\, |U|=k}{\mathbb I}_{{\cal B}_U}.
\end{equation}
Let us upper bound the probability 
$ \PP\{{\cal A}_2\}$.
By the union bound and symmetry, we have
 \begin{align*}
 \PP\{{\cal A}_2\}
 &
 \le 
 \sum_{2\le k\le n/2}\PP\{\eta_k\ge 1\}
 \le
 \sum_{2\le k\le n/2}
 \E \eta_k
 \\
 &
 =
 \sum_{2\le k\le n/2}
 \
 \sum_{U\subset {\cal V}, |U|=k}\PP\{{\cal B}_U\}
 =
 \sum_{2\le k\le n/2}\binom{n}{k}\PP\{{\cal B}_{[k]}\}.
 \end{align*}
 In the second inequality we applied Markov's inequality;
 in the first identity we applied (\ref{Y+});
 in the last identity we used the fact that 
 $\PP\{{\cal B}_U\}=\PP\{{\cal B}_{[k]}\}$ for 
 $|U|=k$.
 Next, using the inequality 
 $\PP\{{\cal B}_{[k]}\}
 \le
  \PP\{{\cal D}_{[k]}\}$ 
  we upper bound
 $\PP\{{\cal A}_2\}\le S_0+S_1+S_2$, 
 where 
 \begin{align}
 \label{2024-03-18+1}
 &
 S_0=\sum_{2\le k\le \varphi_n}
 \binom{n}{k}
 \PP\{{\cal B}_{[k]}\},
\qquad
 S_1=\sum_{\varphi_n< k\le n^{\beta}}
 \binom{n}{k}\PP\{{\cal D}_{[k]}\},
 \\
 \nonumber
 &
 S_2=\sum_{n^{\beta}\le k\le n/2}
 \binom{n}{k}\PP\{{\cal D}_{[k]}\}.
  \end{align}
The sequence $\varphi_n\to+\infty$  as 
$n\to+\infty$ 
will be specified latter.  Now we only mention that 
$\varphi(n)\le \ln n$.
Furthermore, we put 
$\beta
=
\beta_n
 =
 1-\frac{\alpha_n}{2\kappa_n}$. 
 Note that inequality 
 $2\alpha_n\le \kappa_n$ implies 
 $\beta_n\ge \frac{3}{4}$ 
 and (\ref{2025-12-02}), (\ref{2025-12-02+2}) imply  
 $\lim\sup_{n}\beta_n<1$.
 To prove   $\PP\{{\cal A}_2\}=o(1)$ we
 show that
 $S_i=o(1)$ for $i=0,1,2$.

\bigskip

{\it Proof of} $S_0=o(1)$.
Let us evaluate the probabilities 
$\PP\{{\cal B}_{[k]}\}$, $k\ge 2$.
Given $k$, let $T=(V_T,{\cal E}_T)$ be a tree with vertex set 
$V_T=[k]\subset {\cal V}$.  ${\cal E}_T$ stands for the edge set of $T$.
 Fix an integer $r\in\{1,\dots, k-1\}$.
Let 
${\tilde {\cal E}}_T=({\cal E}_T^{(1)},\dots, {\cal E}_T^{(r)})$ be an ordered partition of the set ${\cal E}_T$ 
(every set ${\cal E}_T^{(i)}$ is nonempty, 
${\cal E}_T^{(i)}\cap{\cal E}_T^{(j)}=\emptyset$ 
for $i\not=j$, 
and $\cup_{i=1}^r{\cal E}_T^{(i)}={\cal E}_T$).
We denote by 
$|{\tilde {\cal E}}_T|$ the number of parts (in our case $|{\tilde {\cal E}}_T|=r$).
Let ${\bar t}=(t_1,\dots, t_r)\in [m]^r$ be a 
vector with distinct integer valued coordinates.
We denote by $|{\bar t}|$ the number of coordinates (in our case $|{\bar t}|=r$).
In the special situation, where  $\bar t$ has ordered coordinates,
$t_1<t_2<\cdots<t_r$, we denote such a vector 
${\tilde t}$.
We call 
$({\tilde {\cal E}}_T, {\bar t})$ labeled partition.
By
${\cal T}({\tilde {\cal E}}_T, {\bar t})$  we denote event 
 that ${\cal E}_T^{(i)}\subset {\cal E}_{t_i}$ 
 for each $1\le i\le r$. The 
 event ${\cal T}({\tilde {\cal E}}_T, {\bar t})$ means that the  edges of
$T$ are covered by  the edges of $G_{t_1},\dots G_{t_r}$ so 
that
for every $i$ the edge set ${\cal E}_T^{(i)}$ belongs to the edge set ${\cal E}_{t_i}$ of $G_{t_i}$.
Introduce the set  $H_{\bar t}=[m]\setminus\{t_1,\dots, t_r\}$  and let 
${\cal I}(V_T, H_{\bar t})$ be the event that 
none of the graphs $G_{i}$, $i\in H_{\bar t}$ has an edge connecting some 
$v\in V_T$ and $w\in {\cal V}\setminus V_T$.

\medskip

Let ${\mathbb T}_k$ denote the set of trees on the vertex set $[k]$.
We have, by the union bound and independence 
of $G_{1},\dots, G_{m}$, that
\begin{equation}
\label{suma}
\PP\{{\cal B}_{[k]}\}
\le
\sum_{T\in {\mathbb T}_k}
\sum_{({\tilde {\cal E}}_T,{\tilde t})}
\PP\{{\cal T}({\tilde {\cal E}}_T, {\tilde t})\}
\PP\{{\cal I}(V_T, H_{\tilde t})\}.
\end{equation}
Here the second sum runs over labeled partitions 
$({\tilde {\cal E}}_T, {\tilde t})$
of the edge set of $T$. Note that 
$r:=|{\tilde {\cal E}}_T|=|{\tilde t}|$ runs over the set $\{1,\dots, k-1\}$.
We show in Lemma \ref{q-bounds} (i) that  for  sufficiently large $n$ we have 
for each $1\le k\le n/10$ and each $T\subset {\mathbb T}_k$ 
\begin{align*}
\max_{{\tilde t}:\,|{\tilde t}|\le \phi_n^{-1/4}}
\PP\{{\cal I}(V_T, H_{\tilde t})\}
\le 
e^{
-
k\frac{m}{n}
\left(\kappa_n-\frac{2}{\ln n}\right)
+
k\phi_n^{1/4}
}.
\end{align*}
Here $\phi_n\downarrow 0$ is a sequence that satisfies (\ref{phi+}).
We choose  $\varphi_n$ in (\ref{2024-03-18+1}) such that 
$\varphi_n\phi_n^{1/4}\le 1$. Then for $k\le \varphi_n$ the
right side is at most $e^{
-
k\frac{m}{n}
\left(\kappa_n-\frac{2}{\ln n}\right)
+
1
}$.

Furthermore, we show below that for large $n$ 
\begin{align}
\label{S_T+}
&
\sum_{({\tilde {\cal E}}_T,{\tilde t})}
\PP\{{\cal T}({\tilde {\cal E}}_T,{\tilde t})\}
\le  
k^kc_1\phi_n
\end{align} 
Invoking these inequalities in
 (\ref{suma}) we obtain
\begin{align*}
\PP\{{\cal B}_{[k]}\}
\le 
e^{
-
k\frac{m}{n}
\left(\kappa_n-\frac{2}{\ln n}\right)
+
1
}
\sum_{T\in {\mathbb T}_k}
k^kc_1\phi_n
\le 
e^{
-
k\frac{m}{n}
\left(
\kappa_n-\frac{2}{\ln n}
\right)
+
1}
k^{2k-2}c_1\phi_n.
\end{align*}
In the last step we applied Cayley's formula $|{\mathbb T}_k|=k^{k-2}$.

Now we are ready to show that $S_0=o(1)$. 
Combining the latter inequality with the inequality 
 $\binom{n}{k}\le \frac{n^k}{k!}=\frac{e^{k\ln n}}{k!}$ we estimate
\begin{displaymath}
S_0
\le 
\phi_n c_1e
\sum_{k=2}^{\varphi_n}
e^{k\left(\ln n-\frac{m}{n}\kappa_n+2\frac{m}{n\ln n}\right)}
\frac{k^{2k-2}}{k!}
\end{displaymath}
 Invoking inequalities 
$\frac{m}{n\ln n}\le c_1$ 
and
$\lambda_{n}(0)<c_2$
we upper bound the exponent by $e^{k(2c_1+c_2)}$.
Finally,  we choose a non-decreasing (integer valued) sequence 
$\varphi_n\to+\infty$ as $n\to+\infty$ 
such that $\varphi_n\le \min\{\ln n,\phi_n^{-1/4}\}$ 
and 
$\phi_n
\sum_{k=2}^{\varphi_n}
e^{k(2c_1+c_2)}\frac{k^{2k-2}}{k!}
=o(1)$. Now we have $S_0=o(1)$.

\bigskip

{\it Proof of (\ref{S_T+}).}
Given a tree $T=([k],{\cal E}_T)$ and  partition
${\tilde {\cal E}}_T=({\cal E}_T^{(1)},\dots,{\cal E}_T^{(r)})$, let
$V_T^{(i)}$ be the set of vertices incident to the edges from 
${\cal E}_T^{(i)}$. We denote $e_i=|{\cal E}_T^{(i)}|$ and
$v_i=|V_T^{(i)}|$.
For any labeling 
${\bar t}=(t_{1},\dots, t_{r})$ that assigns labels 
$t_1,\dots, t_r$ to the sets
${\cal E}_T^{(1)},\dots,{\cal E}_T^{(r)}$ we have, by the independence of $G_{1},\dots, G_{m}$,
\begin{displaymath}
\PP\{{\cal T}({\tilde {\cal E}}_T, {\bar t})\}
=
\prod_{i=1}^r
\PP\{{\cal E}_T^{(i)}\subset{\cal E}_{t_i}\}
\le
\prod_{i=1}^r
\PP\{V_T^{(i)}\subset{\tilde {\cal V}}_{t_i}\}
=
\prod_{i=1}^r
\E
\left(
\frac{(X_{t_i})_{v_i}}{(n)_{v_i}}
\right).
\end{displaymath}
We note that the fraction $\frac{(X_{t_i})_{v_i}}{(n)_{v_i}}$ is 
a decreasing function of $v_i$ and it is maximized by 
$\frac{(X_{t_i})_{e_i+1}}{(n)_{e_i+1}}$ since
 we always have $v_i\ge e_i+1$. Indeed,  
 given $|{\cal E}_T^{(i)}|=e_i$ the smallest possible set of vertices
$V_T^{(i)}$ 
corresponds to the configuration of edges of 
${\cal E}_T^{(i)}$ that creates a subtree $(V_T^{(i)},{\cal E}_T^{(i)})\subset T$. Hence $v_i\ge e_i+1$. It follows that 
\begin{align}
\label{Stirling}
\PP\{{\cal T}({\tilde {\cal E}}_T, {\bar t})\}
\le
\prod_{i=1}^r
\E
\left(
\frac{(X_{t_i})_{e_i+1}}{(n)_{e_i+1}}
\right).
\end{align}

Let us evaluate $S_T$.
We have 
\begin{align*}
S_T
=
\sum_{({\tilde {\cal E}}_T,{\tilde t})}
\PP\{{\cal T}({\tilde {\cal E}}_T,{\tilde t})\}
=
\sum_{r=1}^{k-1}
\
\frac{1}{r!}
\sum_{{\tilde {\cal E}}_T:\,|{\tilde {\cal E}}_T|=r}
S({\tilde {\cal E}}_T),
\qquad
S({\tilde {\cal E}}_T)
:=
\sum_{{\bar t}:\, |{\bar t}|=|{\tilde{\cal E}}_T|}
\PP\{{\cal T}({\tilde {\cal E}}_T,{\bar t})\}.
\end{align*}
The last sum runs over the set of 
vectors ${\bar t}=(t_1,\dots t_r)$ having distinct coordinates $t_1,\dots, t_r\in [m]$.
In view of (\ref{Stirling})
the sum $S({\tilde{\cal E}}_T)$ is upper bounded by the sum
\begin{displaymath}
S_*({\tilde{\cal E}}_T)
:=
\sum_{t_1=1}^m\cdots\sum_{t_r=1}^m
\prod_{i=1}^r
\E
\left(
\frac{(X_{t_i})_{e_i+1}}{(n)_{e_i+1}}
\right)
=m^r
\prod_{i=1}^r
\E
\left(
\frac{(X_{i_*})_{e_i+1}}{(n)_{e_i+1}}
\right).
\end{displaymath}
Furthermore, invoking (\ref{phi}) we obtain
\begin{align*}
S_*({\tilde{\cal E}}_T)
\le
\left(
\frac{m}{n}\frac{\phi_n}{\ln n}
\right)^r.
\end{align*}
Next,
using the fact that there are $\frac{(k-1)!}{e_1!\cdots e_r!}$ distinct ordered partitions
${\tilde {\cal E}}_T=({\tilde {\cal E}}_T^{(1)},\dots, {\tilde {\cal E}}_T^{(r)})$
with $|{\tilde {\cal E}}_T^{(1)}|=e_1,
\dots, |{\tilde {\cal E}}_T^{(r)}|=e_r$, we upper bound
\begin{align*}
S_T
\le 
\sum_{r=1}^{k-1}
\sum'_{e_1+\cdots+e_r=k-1}
\frac{(k-1)!}{e_1!\cdots e_r!} 
\left(
\frac{m}{n}\frac{\phi_n}{\ln n}
\right)^r
\le 
\sum_{r=1}^{k-1}
\left(
\frac{m}{n}\frac{\phi_n}{\ln n}
\right)^r
r^{k-1}.
\end{align*}
Here the sum $\sum'_{e_1+\dots+e_r=k-1}$ runs over the set of vectors
$(e_1,\dots, e_r)$ having integer valued coordinates $e_i\ge 1$ 
satisfying $e_ 1+\cdots+e_r=k-1$. 
Hence, $\sum'_{e_1+\cdots+e_r=k-1}
\frac{(k-1)!}{e_1!\cdots e_r!} 
\le r^{k-1}$.
Finally, invoking inequality $\frac{m}{n\ln n}\le c_1$ 
and using $\sum_{i=1}^{k-1}r^{k-1}\le k^k$ we obtain (\ref{S_T+}),
\begin{align*}
S_T
\le
k^k\max_{1\le r\le k-1}\left(c_1\phi_n\right)^r
\le 
k^kc_1\phi_n,
\end{align*}
where in the last step we used inequality $c_1\phi_n<1$,
which holds for large $n$, because $\phi_n\downarrow 0$.

\bigskip

{\it Proof of} $S_1=o(1)$. Lemma \ref{q-bounds} (ii) implies  
$
\PP\{{\cal D}_{[k]}\}
\le e^{-k\frac{m}{n}\left(\kappa_n-\frac{2}{\ln n}\right)}$ for $\varphi_n\le k\le n^{\beta}$.
Using this inequality and the inequality $\binom{n}{k}\le \frac{n^k}{k!}=\frac{e^{k\ln n}}{k!}$ we estimate
\begin{displaymath}
S_1
\le
\sum_{ \varphi_n< k\le n^{\beta}}
\frac{1}{k!}
e^{k
\left(
\lambda_{n}(0)+2\frac{m}{n\ln n}
\right)}
\le 
\sum_{ \varphi_n< k\le n^{\beta}}
\frac{1}{k!}
e^{k
(c_2+2c_1)}.
\end{displaymath}
The quantity on the right is $o(1)$ 
 because the series  $\sum_k\frac{1}{k!}
e^{k(c_2+2c_1)}$ converges 
and  $\varphi_n \to+\infty$.

\bigskip

{\it Proof of} $S_2=o(1)$. 
Lemma \ref{q-bounds} (iii) implies  
$
\PP\{{\cal D}_{[k]}\}
\le 
e^{-2\alpha_nk\frac{m}{n}\frac{n-k}{n-1}}
\le e^{-\alpha_nk\frac{m}{n}}$ for 
$n^{\beta}\le k\le n/2$. 
We will use the inequality  (see formula (18) in \cite{Daumilas_Mindaugas2023})
\begin{align}
\label{2026-01-30}
\binom{n}{k}
&
\le 
e^{2k+(1-\beta)k\ln n},
\qquad
\
{\text{for}}
\qquad
n^{\beta}\le k\le n/2.
\end{align}
Combining these inequalities and using identity 
$\kappa_n\frac{m}{n}=\ln n-\lambda_{n}(0)$
we estimate
\begin{displaymath}
\binom{n}{k}
\PP\{{\cal D}_{[k]}\}
\le 
e^{2k+(1-\beta)k\ln n-\alpha_nk\frac{\ln n-\lambda_{n,m}}{\kappa_n}}
=
e^{-k\frac{\alpha_n}{2\kappa_n}\ln n+kR_n},
\end{displaymath}
where $R_n=2+\lambda_{n}(0)\frac{\alpha_n}{\kappa_n}$.
Since 
$\frac{1}{2}\ge \limsup_n\frac{\alpha_n}{\kappa_n}\ge \liminf_n\frac{\alpha_n}{\kappa_n}>0$ 
(the latter inequality follows from (\ref{2025-12-02}), (\ref{2025-12-02+2}))  and $\lambda_{n}(0)\le c_2$ for some $c_2$, we conclude that $R_2$ is bounded from above by a constant and
$\sum_{k\ge n^{\beta}} e^{-k\frac{\alpha_n}{2\kappa_n}\ln n+kR_n}=o(1)$ as $n\to+\infty$. Hence  
$S_2=o(1)$.
 \end{proof}

\begin{proof}[Proof of Lemma \ref{Lemma_2}]
We first evaluate factorial moments $\E(\eta_1)_r$, $r=1,2,\dots$.
using  the identity
\begin{displaymath}
\binom{\eta_1}{r}
=
\sum_{\{v_{i_1},\dots, v_{i_r}\}\subset {\cal V}}
{\mathbb I}_{\{d(v_{i_1})=0\}}\cdots{\mathbb I}_{\{d(v_{i_r})=0\}},
\end{displaymath}
where  both sides count the (same) number of subsets of 
${\cal V}$ of size $r$ consisting of vertices having degree $0$. 
Fix $r$. Taking the expected values of both sides we obtain, by symmetry,
\begin{align}
\nonumber
\E (\eta_1)_r
&
=r!\E\left( \sum_{\{v_{i_1},\dots, v_{i_r}\}\subset {\cal V}}{\mathbb I}_{\{d(v_{i_1})=0\}}\cdots{\mathbb I}_{\{d(v_{i_r})=0\}}\right)
\\
\label{2023-10-24+2}
&
=
r!\binom{n}{r}\PP\{d(v_1)=0,\dots, d(v_r)=0\}.
\end{align}
%
Now we analyse the product $\prod_{i=1}^mP_{i}$ with
$P_{i}:=\PP\{d_i(v_1)=0,\dots, d_i(v_r)=0\}$.
To this aim we apply inclusion-exclusion inequalities to the probability 
$\PP\{\cup_{k=1}^r\{d_i(v_k)>0\}\}=1-P_{i}$. We have
\begin{align*}   
S_1-S_2 \le 1-P_{i}\le  S_1,
\end{align*}
where
\begin{align*}
 S_1
 &
 =\sum_{k=1}^r\PP\{d_i(v_k)>0\}=r\PP\{d_i(v_1)>0\}
 =
r\frac{\E X_{i}}{n},
\\
S_2
&
=
\sum_{1\le k<j\le r}
\PP\{d_i(v_k)>0,\, d_i(v_j)>0\}
=
\binom{r}{2}
\PP\{d_i(v_1)>0,\, d_i(v_2)>0\}
\\
&
=
\binom{r}{2}
\frac{\E (X_{i})_2}{(n)_2}.
\end{align*}
It follows that
\begin{align*}
P_{i}=1-\frac{r}{n}\E X_{i}
+
\theta_i\frac{(r)_2}{(n)_2}\E (X_{i})_2
=:1-a_i+b_i,
\end{align*}
with some  $\theta_i\in[0,1]$.
Note that   relation
(\ref{phi+}) implies
\begin{align}
\label{2025-11-20+2}
\max_{1\le i\le m}a_i=o(1) 
\qquad
{\text{and}}
\qquad
 \max_{1\le i\le m}b_i=o(1).
 \end{align}
Using $1+t=e^{\ln(1+t)}$ and $t-t^2\le \ln(1+t)\le t$ (these inequalities hold for $|t|\le 0.5$ at least)
we write $P_{i}$ in the form
\begin{align}
\label{2025-11-20+4}
P_{i}
=
e^{\ln(1-a_i+b_i)}
=
e^{-a_i+b_i-R_i}
=
e^{-a_i+R'_i},
\end{align}
where $0\le R_i\le (-a_i+b_i)^2$ and 
where
$R'_i=b_i-R_i$ satisfies
$|R'_i|\le b_i+2a_i^2+2b_i^2$.
Note that for large $n$ we have $b_i\le 1$, by  (\ref{2025-11-20+2}). Hence  
$b_i^2< b_i$. Now using (\ref{phi+})  we upper bound the sum
\begin{align}
\label{2025-11-20+5}
\sum_{i=1}^m
|R'_i|
\le 
2\sum_{i=1}^m a^2_i
+
3\sum_{i=1}^m |b_i|
=o(1).
\end{align}
Indeed,  we have
\begin{align*}
\sum_{i=1}^m a^2_i
&
=
\frac{r^2}{n^2}
\sum_{i=1}^m
(\E X_{i})^2
\le 
\frac{r^2}{n^2}
\sum_{i=1}^m
\E X^2_{i}
=\frac{r^2}{n^2}
\
o\left(\frac{mn}{\ln n}\right)
=
o(1),
\\
\sum_{i=1}^m |b_i|
&
\le 
\frac{(r)_2}{(n)_2}
\sum_{i=1}^m
\E X^2_{i}
=
\frac{(r)_2}{(n)_2}
\
o\left(\frac{mn}{\ln n}\right)
=o(1)
\end{align*}
Here in the last steps we used $m=O(n\ln n)$.

Finally, combining (\ref{2025-11-20+4}), (\ref{2025-11-20+5}) we evaluate the product
\begin{align*}
\prod_{i=1}^mP_{i}
=
e^{-\sum_ia_i+o(1)}
=
e^{-r\frac{m}{n}\kappa_n+o(1)}.
\end{align*}
Invoking this expression in (\ref{2023-10-24+2})
we obtain an approximation to the factorial moment
\begin{align}
\label{2025-11-20+7}
\E (\eta_1)_r
=
(n)_r \prod_{i=1}^mP_{i}
=
e^{r\ln n-r\frac{m}{n}\kappa_n}
(1+o(1))
=
e^{r\lambda_{n}(0)}
(1+o(1)).
\end{align}

Let us prove (\ref{2025-11-20+6}).
For $\lambda_{n}(0)\to-\infty$  
relation (\ref{2025-11-20+7}) implies
$\E \eta_1=o(1)$. Hence
$\PP\{\eta_1\ge 1\}=o(1)$, 
by Markov's inequality.
For $\lambda_{n}(0)\to+\infty$ relation (\ref{2025-11-20+7}) implies
$\E \eta_1\to+\infty$ and 
\begin{align*}
\E \eta^2_1-(\E \eta_1)^2
=
\E(\eta_1)_2+\E \eta_1-(\E \eta_1)^2
=
o(e^{2\lambda_{n}(0)})
=o((\E \eta_1)^2).
\end{align*}
Now Chebyshev's inequality implies
\begin{align*}
\PP\{\eta_1=0\}\le \PP\{|\eta_1-\E \eta_1|\ge \E \eta_1\}
\le
\frac{\E \eta_1^2-(\E \eta_1)^2}{(\E \eta_1)^2}
=
o(1).
\end{align*}
For $\lambda_{n}(0)\to c$ relation 
(\ref{2025-11-20+7}) implies $\E(\eta_1)_r\to e^{rc}$.
Note that $e^{rc}$ is the $r$-th factorial moment of the Poisson distribution with parameter $e^c$. Now the convergence of the distribution of $\eta_1$ to the Poisson distribution follows by the method of moments. An immediate consequence of this convergence is  $\lim_{n}\PP\{\eta_1=0\}=e^{-e^c}$. 
\end{proof}

\medskip
Example 2 (continued). Here we verify conditions (\ref{2025-11-24}), (\ref{2025-12-02}) and show that (\ref{T1+}) remains true with $\lambda_{n}(0)$  replaced by $\lambda_{n}'$.
 Recall that $X_{n,i}$ is the number of isolated vertices in $G'_{n,i}$. From the fact that random variables $X_{n,1},\dots, X_{n,m}$ are identically distributed we conclude that random variables $X_{n,i_*}$ and $X_{n,1}$ have the same distribution.  We use  this observation in the proof of (\ref{2025-11-24}), (\ref{2025-12-02}). To show (\ref{2025-11-24}), we invoke inequality $X_{n,1}\le Y_1$ and the expression for the conditional expectation
$\E(X_{n,1}|Y_1,Q_1)={\hat Y}_{n,1}h({\hat Y}_{n,1},Q_1)$.
 We have
 \begin{align*}
 \E 
 \left(
 X_{n,i_*}\ln(1+X_{n,i_*}){\mathbb I}_{\{X_{n,i_*}>t\}}
 \right)
&
 =
 \E 
 \left(
 X_{n,1}\ln(1+X_{n,1}){\mathbb I}_{\{X_{n,1}>t\}}
 \right)
 \\
 &
 \le 
 \E 
 \left(
 X_{n,1}\ln(1+Y_1){\mathbb I}_{\{Y_1>t\}}
 \right)
 \\
 &
 =
 \E
 \left( 
 \E 
 \left(
 X_{n,1}\ln(1+Y_1){\mathbb I}_{\{Y_1>t\}}|Y_1,Q_1
 \right)
 \right)
 \\
 &
 =
 \E 
 \left(
  {\hat Y}_{n,1}h({\hat Y}_{n,1},Q_1)\ln(1+Y_1){\mathbb I}_{\{Y_1>t\}}
  \right)
 \\
 &
 \le 
 \E 
 \left(
 Y_1h(Y_1,Q_1)\ln(1+Y_1){\mathbb I}_{\{Y_1>t\}}
 \right).
 \end{align*}
Our asumption 
$
\E 
 \left(
 Y_1h(Y_1,Q_1)\ln(1+Y_1)
 \right) 
 <\infty
 $
implies
 $\E 
 \left(
 Y_1h(Y_1,Q_1)\ln(1+Y_1){\mathbb I}_{\{Y_1>t\}}
 \right)=o(1)$ as $t\to+\infty$. We arrived to 
 (\ref{2025-11-24}).
To show (\ref{2025-12-02}) we evaluate the probability
\begin{align*}
\PP\{X_{n,i_*}>0\}
=
\PP\{X_{n,1}>0\}
=
\E \left(1-(1-Q_1)^{{\hat Y}_{n,1}({\hat Y}_{n,1}-1)/2}\right).
\end{align*}
This probability convereges as $n\to+\infty$ to 
$\E \left(1-(1-Q_1)^{Y_{1}(Y_{1}-1)/2}\right)$. The latter quantity is positive because $\PP\{Y_1\ge 2, Q_1>0\}>0$. Hence 
  (\ref{2025-12-02}) holds. 
Finally, we show that $\kappa_{n}'-\kappa_{n}=o\left(\frac{1}{\ln n}\right)$. To this aim
we write $\kappa_n$ in the form
\begin{align*}
\kappa_n
&
=
\E X_{n,i_*}
=
\E X_{n,1}
=
\E 
\left(
{\hat Y}_{n,1}h({\hat Y}_{n,1},Q_1)
\right)
\\
&
=
\E 
\left(
Y_1
h(Y_1,Q_1)
{\mathbb I}_{\{Y_1\le n\}}
\right)
+
\E 
\left(nh(n,Q_{1}){\mathbb I}_{\{Y_1>n\}}
\right)
\end{align*}
and evaluate the difference
\begin{align*}
0
\le 
\kappa_n'-\kappa_n
&
\le 
\E 
\left(
Y_1h(Y_1,Q_1)
\right)
-
\E 
\left(
Y_1h(Y_1,Q_1){\mathbb I}_{\{Y_1\le n\}}
\right)
\\
&
=
\E 
\left(Y_{1}h(Y_{1},Q_{1}){\mathbb I}_{\{Y_1> n\}}
\right) 
 \\
 &
 \le 
 \frac{1}{\ln(1+n)}
 \E 
\left(
Y_1h(Y_1,Q_1)\ln(1+Y_1){\mathbb I}_{\{Y_1> n\}}
\right) 
\\
&
=\frac{1}{\ln (1+n)}o(1).
\end{align*}

\subsection{Proof of Theorem \ref{T2}}

The scheme of the proof of Theorem \ref{T2} is similar to that of Theorem \ref{T1}: we establish
 expansion property (Lemma \ref{L2.2}) and  show
  concentration of vertices of 
degree $ak$  (Lemma \ref{L1} and  Corollary 
\ref{col2}).

The section is organized as follows. We first state Lemmas 
 \ref{L2.2},
 \ref{L1}
  and Corollary 
\ref{col2}. Then we prove Theorem \ref{T2}. Afterwards we prove 
Lemma 
\ref{L2.2}, \ref{L1} and Corollary 
\ref{col2}.

\begin{lem}\label{L2.2}
 Let $k\ge 1$ be an integer. Let $n,m\to+\infty$.
  Assume that 
 $\lambda(0)\to-\infty$ and $n\ln n\asymp m$.
 Assume that 
 (\ref{2025-12-02}),
(\ref{a}) hold.
   Assume that $\sup_n\E X_{n,i_*}^{k+2}<\infty$. Then 
 $
 \PP\{{\cal B}_{k}\cap \{G_{[n,m]}\in {\cal C}_1\}\cap {\cal P}_k\}=o(1)
 $.
\end{lem}

\begin{lem}\label{L1} Let $a\ge 1$ and $k\ge 0$ be integers. Let $n,m\to+\infty$. Assume that
$n\ln n\asymp m$. Assume that (\ref{a}) holds. Assume that
\begin{align}
\label{condition_X^2}
\limsup_n\E X_{n,i_*}^2<\infty.
\end{align}
Then 
$\PP\{{\cal A}\}=o(1)$
and 
$\PP\{N'_k\ge 1\}=o(1)$.
Furthermore, 
for $\lambda_n(k)\to+\infty$
we have
$\E N_k\to+\infty$, $\E N_{*k}\to+\infty$, and
$N_{*k}=(1+o_P(1))\E N_{*k}$.
For $\lambda_n(k)\to-\infty$
we have $\E N_k\to 0$, $\E N_{*k}\to 0$, and
consequently $\PP\{N_k\ge 1\}=o(1)$ and $\PP\{N_{*k}\ge 1\}=o(1)$.
\end{lem}

\begin{col}\label{col2}
Let $a\ge 1$ and $k\ge 1$ be integers. 
Let $n,m\to+\infty$.  Asume that (\ref{a}),
(\ref{condition_X^2}) hold.
For 
$m,n$ 
satisfying 
\begin{align}
\label{2026-01-14+2}
\ln n
+
(k-1)\ln\ln n
\prec
\kappa \frac{m}{n}
\prec
\ln n
+
k\ln\ln n
\end{align}
 the following 
properties hold whp: 

(i) for  $0\le r\le k-1$ we have $N_r=0$;

(ii) for each  (fixed) $r\ge k$ we have that $N_{*r}\to \infty$ and each vertex $v$ with $d'(v)=r$ is the center of an $r$-blossom;

(iii) the minimal degree $\delta(G_{[n,m]})=ka$. 
\end{col}

{\it Proof of Theorem} \ref{T2}. 
Proof of (\ref{2026-01-20+1}).
For $k=0$ Lemma \ref{L1} shows
$N_{*0}\ge 1$ whp. The simple identity $N_0=N_{*0}$ implies $N_0\ge 1$ whp. Hence $G_{[n,m]}$ contains an  isolated vertex whp. Therefore $\PP\{G_{[n,m]}\in {\cal C}_1\}=o(1)$. 

For $k\ge 1$ Lemma \ref{L1} implies
$N_{*k}\ge 1$ and $N'_{k}=0$ whp.
 Hence $G_{[n,m]}$ contains a $k$-blossom with each  petal contributing  $a$  unique neighbours to the central vertex of the blossom. Therefore 
 the central vertex has degree $ak$ in $G_{[n,m]}$. Removal of 
 its $ak$ neighbours makes this vertex isolated. We conclude 
 that
  $\PP\{G_{[n,m]}\in{\cal C}_{ak+1}\}=o(1)$. 

Proof of (\ref{2025-08-25+r}).  It suffices to prove  
(\ref{2025-08-25+r}) for $m=m(n)$ satisfying (\ref{2026-01-21}).
 Indeed  
for a sequence $m(n)$ satisfying 
$\ln n+k\ln\frac{m}{n}\prec\kappa\frac{m}{n}$
we can find a sequence $m'(n)$ satisfying 
(\ref{2026-01-21}) and
such that $m'(n)\le m(n)$. We may assume that the first $m'$ communities $G_1,\dots, G_{m'}$ satisfy (\ref{a}).
Since $G_{[n,m]}$ can be obtained from $G_{[n,m']}$ by adding $m-m'$ communities $G_{m'+1},\dots, G_{m}$ there is a natural coupling 
$G_{[n,m']}\subset G_{[n,m]}$ with probability $1$. 
Hence, $\PP\{G_{[n,m']}\in {\cal C}_{(k+1)a}\} \le 
\PP\{G_{[n,m]}\in {\cal C}_{(k+1)a}\}$. Now
relation $\PP\{G_{[n,m']}\in {\cal C}_{(k+1)a}\}=1-o(1)$ implies $\PP\{G_{[n,m]}\in {\cal C}_{(k+1)a}\}=1-o(1)$.

For the rest of the proof we assume that 
 (\ref{2026-01-21}) 
 holds. 
 Note that (\ref{2026-01-21}) implies 
$m=O(n\ln n)$. We consider the cases $k=0$ and $k\ge 1$ separately.

Let $k=0$.  Theorem \ref{T1}  shows
$\PP\{G_{[n,m]}\in {\cal C}_1\}=1-o(1)$. For $a=1$ nothing more needs to be proven.
For $a\ge 2$
each vertex of a  connected graph $G_{[n,m]}$  
has degree at least $a$. We conclude 
that $\PP\{{\cal P}_{a-1}\}=1-o(1)$. Next, we invoke the bound
$\PP\{{\cal B}_{a-1}\cap {\cal C}_1\cap{\cal P}_{a-1}\}=o(1)$,
 which is shown in   Lemma \ref{L2.2}. Combining these bounds we arrive to
  (\ref{2025-08-25+r})
\begin{align*}
\PP\{G_{[n,m]}\not\in {\cal C}_{a}\}
&
=
\PP\{G_{[n,m]}\not\in {\cal C}_{a}\cap \{G_{[n,m]}\in {\cal C}_1\}
\cap{\cal P}_{a-1}\}+o(1)
\\
&
=
\PP\{{\cal B}_{a-1}\cap \{G_{[n,m]}\in {\cal C}_1\}\cap{\cal P}_{a-1}\}
+
o(1)
\\
&
=
o(1).
\end{align*}

Let $k\ge 1$. Corollary \ref{col2} (iii) implies  
$\delta(G_{[n,m]})= (k+1)a$ whp.
Hence, $\PP\{{\cal P}_{(k+1)a-1}\}=1-o(1)$. 
Furthermore,  (\ref{2026-01-21}) implies
$\PP\{G_{[n,m]}\in {\cal C}_1\}=1-o(1)$, 
by Theorem \ref{T1}. Moreover,
Lemma \ref{L2.2} implies  
\begin{align*}
\PP\{
{\cal B}_{(k+1)a-1}
\cap \{G_{[n,m]}\in {\cal C}_1\}
\cap{\cal P}_{(k+1)a-1}
\}
=
o(1).
\end{align*}
Combining these bounds we 
we obtain
\begin{align*}
\PP\{G_{[n,m]}\not\in {\cal C}_{a(k+1)}\}
&
=
\PP\{\{G_{[n,m]}\not\in {\cal C}_{a(k+1)}\}
\cap 
\{G_{[n,m]}\in {\cal C}_1\}
\cap{\cal P}_{a(k+1)-1}\}+o(1)
\\
&
=
\PP\{{\cal B}_{a(k+1)-1}
\cap 
\{G_{[n,m]}\in {\cal C}_1\}
\cap
{\cal P}_{a(k+1)-1}\}
+
o(1)
\\
&
=
o(1).
\end{align*}  
The proof of  (\ref{2025-08-25+r}) is complete.
Finally,   (\ref{degree}) follows by  Corollary \ref{col2} (iii).
\qed

\begin{proof}[Proof of Lemma \ref{L2.2}]
We write $G=G_{[n,m]}$ for short.
Note that
 (\ref{2025-12-02}) implies  that 
 for some 
 ${\tilde \alpha}>0$  and all  sufficiently large $n$
 we have $\alpha_n>{\tilde \alpha}$. We assume below that 
 $\alpha_n>{\tilde \alpha}$.
Let $p_{s,r}$ denote the probability that 
$\{s+1,\dots, s+r\}$ 
induces  a component in $G-\{1,\dots, s\}$
and each vertex $i\in\{1,\dots, s\}$  is linked to some vertex from 
$\{s+1,\dots, s+r\}$
in $G$. 
Let $p^*_{s,r}$ denote the probability that 
$G-\{1,\dots, s\}$ has no edges connecting 
$\{s+1,\dots, s+r\}$  and $[n]\setminus [s+r]$. Note  that 
$p_{s,r}\le p^*_{s,r}$. 
We have, by the union bound and symmetry, that

\begin{align}
\label{2025-07-24+r}
&
\PP\{{\cal B}_{k}\cap \{G_{[n,m]}\in{\cal C}_1\}\cap {\cal P}_k\}
\le 
\sum_{s=1}^{k}\binom{n}{s}
\sum_{2\le r\le (n-s)/2}
\binom{n-s}{r}p_{s,r}
\le S_1+S_2,
\end{align}
where 
\begin{align*}
S_1
:=
\sum_{s=1}^{k}\binom{n}{s}
\sum_{2\le r\le n^{\beta}}
\binom{n-s}{r}p_{s,r},
\qquad
S_2:=
\sum_{s=1}^{k}\binom{n}{s}
\sum_{n^{\beta} < r\le (n-s)/2}
\binom{n-s}{r}p^*_{s,r}.
\end{align*}
We explain inequality  (\ref{2025-07-24+r}): $s$ stands for the size of the 
minumal vertex cut,  $r$ stands for the size of the  smallest  component  of the graph
with a  minimal cut set removed.  Given $1\le s\le k$ there are
$\binom{n}{s}$ ways to select the cut set of size $s$. Furthermore, there are $\binom{n-s}{r}$ ways to select the vertex set of the component of size $r$ from the remaining $n-s$ vertices.
 We also use the fact that on the event ${\cal P}_k$  the minimal component size $r$ is at least $2$.
We choose 
$\beta=1-\frac{{\tilde \alpha}}{2\kappa}$ and
show that $S_i=o(1)$ for $i=1,2$.

\smallskip

{\it Proof of} $S_1=o(1)$.
Given $s$ and $r$ we evaluate the probability $p_{s,r}$.
We begin by introducing some new notation.
Denote $S=[s]$, $U=[s+r]\setminus[s]$. 
We think of $S$ as a potential minimal cut set  and $U$ as the smallest component of $G-S$.
Let ${\tilde S}=\{S_1,\dots, S_h\}$ be a partition of $S$ into disjoint non-empty parts, $S=S_1\cup\cdots\cup S_h$. We denote $s_i=|S_i|$. 
Given ${\tilde t}=(t_1,\dots, t_h)\in [m]^h$
such
that $t_1<\cdots<t_h$ and a permutation $\pi:[h]\to[h]$
define the event 
\begin{align*}
{\cal F}({\tilde S}, \pi, {\tilde t})
=\{S_{i}\subset {\tilde {\cal V}}_{t_{\pi(i)}},  \,\forall  1\le i\le h\}\cap
\{ u_i\in {\tilde {\cal V}}_{t_\pi(i)} {\text{\ for \ some \ }} u_i\in U,\  \forall  1\le i\le h\},
\end{align*}
which holds when vertices of $S_i\in{\tilde S}$
 are connected to some vertex $u_i \in U$ by the edges of community 
 $G_{t_{\pi(i)}}$, for $i=1,\dots, h$. 
We denote by  $G_U$ the subgraph of $G$ induced by vertex set $U$. Note that $G_U$  is induced by $U$ in $G-S$ as well. In the case where $S$ is a minimal cut the subgraph 
$G_U$ is connected and it contains at least two vertices. Hence $G_U$ contains at least one edge. 
Such an edge can be produced either by  some 
$G_{t_{\pi(\ell)}}$ (configuration referred to as  case (i)),  or  by some $G_j$ distinct from $G_{t_1},\dots, G_{t_h}$ (configuration referred to as  case (ii)).
In the case (i) at least one of the events
 \begin{align}
 \label{2025-08-12+r}
 {\cal F}_{\ell}({\tilde t})
 =\{v_1,v_2\in U\cap {\tilde {\cal V}}_{t_\pi(\ell)} {\text{ \ for\ some\ }}v_1,v_2\in U \},
 \qquad 
  1\le \ell\le h
\end{align}
occurs. In the case (ii) at least one of the events
 \begin{align*}
 {\cal F}_j
 =
\{v_1,v_2\in U\cap {\tilde {\cal V}}_{j} {\text{ \ for\ some\ }}v_1,v_2\in U\},
\qquad
 j\in [m]\setminus \{t_1,\dots, t_h\}
\end{align*}
occurs.
Next, given  $H\subset [m]$, we introduce event ${\cal I}(S,U, H)$
 that 
none of the communities $G_j$, $j\in H$ has an edge connecting (some 
$v\in$) $U$ with (some $w\in$) ${\cal V}\setminus (S\cup U)$.
Let ${\cal I}_U$ denote the event that  $G_U$  is connected.
Let ${\mathbb S}_h$ denote the collection of partitions of $S$ into $h$ non-empty parts. 

\smallskip

Let us upper bound the probability  $p_{s,r}$. We have, by the union bound, that
\begin{equation}
\label{2025-08-01+r}
p_{s,r}
\le
\sum_{h=1}^s
\sum_{{\tilde S}\in {\mathbb S}_h}
\sum_{{\tilde t}\in T_h}
\sum_{\pi:[h]\to[h]}
\PP\{{\cal F}({\tilde S}, \pi, {\tilde t})\cap {\cal I}(S,U,[m])
\cap
{\cal I}_U\}.
\end{equation}
Here $T_h$ stands for the set of vectors ${\tilde t}=(t_1,\dots, t_h)\in[m]^h$ with
$t_1<\cdots<t_h$.
Now we will
evaluate probabilities on the right of 
(\ref{2025-08-01+r}). 
We fix $h$, ${\tilde S}$, ${\tilde t}$ and $\pi$.
We recall that the connectivity of $G_U$ implies that at least one of the events 
$\bigcup_{\ell\in[h]}{\cal F}_{\ell}({\tilde t})$ 
and 
$\bigcup_{j\in[m]\setminus\{t_1,\dots, t_h\}}
{\cal F}_j$ holds. Hence, by the union bound,
\begin{align*}
\PP\left\{
{\cal F}
({\tilde S}, \pi, {\tilde t})
\cap 
{\cal I}(S,U,[m])
\cap
{\cal I}_U
\right\}
&
\le
\PP\left\{
{\cal F}({\tilde S}, 
\pi, 
{\tilde t})
\cap 
{\cal I}(S,U,[m])
\cap
\left(
\bigcup_{\ell\in[h]}
{\cal F}_{\ell}({\tilde t})
\right)
\right\}
\\
&
+
\PP\left\{
{\cal F}({\tilde S}, \pi, {\tilde t})\cap {\cal I}(S,U,[m]) 
\cap
\left(
\bigcup_{j\in[m]\setminus\{t_1,\dots, t_h\}}
{\cal F}_j
\right)\right\}.
\end{align*}
Using  the independence of $G_1,\dots, G_m$  and the observation that  probabilities on the right increase (at least nondecrease) if 
we replace event ${\cal I}(S,U,[m])$ 
by 
${\cal I}(S,U,[m]\setminus\{t_1,\dots, t_h\})$ 
or by ${\cal I}(S,U,[m]\setminus\{t_1,\dots, t_h, j\})$
we 
obtain
\begin{align}
\label{2025-08-13+1+r}
&\PP\left\{{\cal F}({\tilde S}, \pi, {\tilde t})\cap {\cal I}(S,U,[m])
\cap
{\cal I}_U
\right\}
\le
\\
\nonumber
&
\qquad
\
\quad
\sum_{\ell\in [h]}\PP\left\{{\cal F}({\tilde S}, \pi, {\tilde t})
\cap
{\cal F}_{\ell}({\tilde t})\right\}
\PP\left\{{\cal I}(S,U,[m]\setminus\{t_1,\dots, t_h\}\right\}
\\
\nonumber
&
\qquad
+
\sum_{j\in[m]\setminus\{t_1,\dots, t_h\}}
\PP\left\{
{\cal F}({\tilde S}, \pi, {\tilde t})\cap{\cal F}_j\right\}
\PP\left\{ {\cal I}(S,U,[m]\setminus\{t_1,\dots, t_h,j\}\right\}.
\end{align}

Furthermore,  we estimate the probabilities
\begin{align}
\label{2025-08-13+2+r}
\PP\left\{{\cal I}(S,U,[m]\setminus\{t_1,\dots, t_h\}\right\}
&
\le
\PP\left\{ {\cal I}(S,U,[m]\setminus\{t_1,\dots, t_h,j\}\right\}
\\
\nonumber
&
\le
e^{-r\left(\kappa\frac{m}{n}-o(1)\right)},
\end{align}
where the error term  bound $o(1)$ holds uniformly in $r$, $j$ and
${\tilde t}\in T_h$ with $h\le k$.
The first inequality of (\ref{2025-08-13+2+r}) is obvious, the last one follows by Lemma \ref{q-bounds} (i).
 
Now, we estimate  probabilities
$\PP\left\{{\cal F}({\tilde S}, \pi, {\tilde t})
\cap
{\cal F}_{\ell}({\tilde t})\right\}$ 
and
$\PP\left\{{\cal F}({\tilde S}, \pi, {\tilde t})
\cap
{\cal F}_j\right\}$.
 We have
\begin{align*}
\PP\left\{{\cal F}({\tilde S}, \pi, {\tilde t})
\cap
{\cal F}_{\ell}({\tilde t})\right\}
\le
r^{h-1}\binom{r}{2}
\frac{\E(X_{t_{\pi(\ell)}})_{s_{\ell}+2}}
{(n)_{s_{\ell}+2}}
\prod_{i\in[h]\setminus\{\ell\}}
\frac{\E(X_{t_{\pi(i)}})_{s_i+1}}{(n)_{s_i+1}}.
\end{align*}

Here $\binom{r}{2}$  counts potential vertex pairs $\{v_1,v_2\}$ that 
realize event 
${\cal F}_{\ell}({\tilde t})$  and
$\frac{\E(X_{t_{\pi(\ell)}})_{s_{\ell}+2}}
{(n)_{s_{\ell}+2}}$ is the probability that the random set
${\tilde {\cal V}}_{t_{\pi(\ell)}}$ covers the union 
$S_{\ell}\cup\{v_1,v_2\}$; 
$r^{h-1}$ counts 
$(h-1)$-tuples of vertices $u_i$
 such that for every $i\in[h]\setminus\{\ell\}$ the random set 
 ${\tilde {\cal V}}_{t_{\pi(i)}}$ covers the union  $S_i\cup\{u_i\}$; the ratios 
$\frac{\E(X_{t_{\pi(i)}})_{s_i+1}}{(n)_{s_i+1}}$
evaluate  probabilities of such covers.
Using simple 
inequality $\frac{(x)_j}{(n)_j}\le \frac{x^j}{n^j}$ 
(valid 
for $n\ge x$) we estimate 
$\frac{\E(X_r)_{j}}
{(n)_j}\le \frac{\E X_r^j}{n^j}$. 
We obtain
\begin{align}
\label{2025-08-06+r}
\PP\left\{{\cal F}({\tilde S}, \pi, {\tilde t})
\cap
{\cal F}_{\ell}({\tilde t})\right\}
&
\le
\frac{r^{h+1}}{2}
\frac{1}
{n^{h+1+\sum_{i\in [h]}s_i}}
\E X_{t_{\pi(\ell)}}^{s_{\ell}+2}
\prod_{i\in[h]\setminus\{\ell\}}
\E X_{t_{\pi(i)}}^{s_i+1}
\\
\nonumber
&
\le \frac{1}{2}
\frac{r^{h+1}}{n^{h+1+s}}
\prod_{i\in[h]}
\E X_{t_i}^{s+2}.
\end{align}
In the last step  we used $1<\E X_i^f<\E X_i^g$
 for $0<f<g$ (recall that $\PP\{X_i\ge 2\}=1$).

Next, using the independence of $G_j,G_{t_1},\dots, G_{t_h}$
 we similarly estimate the probability
\begin{align}
\nonumber
\PP\left\{{\cal F}({\tilde S}, \pi, {\tilde t})
\cap
{\cal F}_j\right\}
&
=
\PP\left\{
{\cal F}_j\right\}
\times
\PP\left\{{\cal F}({\tilde S}, \pi, {\tilde t})
\right\}
\\
\nonumber
&
\le
\frac{\E(X_j)_{2}}{(n)_{2}}
\times
\left(r^{h}
\prod_{i\in[h]}
\frac{\E(X_{t_{\pi(i)}})_{s_i+1}}{(n)_{s_i+1}}
\right)
\\
\nonumber
&
\le
r^h\frac{1}{n^{h+s+2}} (\E X_j^2)\prod_{i\in[h]}\E X_{t_{\pi(i)}}^{s_i+1}
\\
\label{2025-08-13+r}
&
\le
\frac{r^h}{n^{h+s+2}} (\E X_j^2)
\prod_{i\in[h]}\E X_{t_i}^{s+1}.
\end{align}

Finally, we invoke (\ref{2025-08-13+2+r}), (\ref{2025-08-06+r}),  
(\ref{2025-08-13+r}) in (\ref{2025-08-13+1+r})
and apply inequalities
\begin{align*}
\prod_{i\in[h]}\E X_{t_i}^{s+1}
\le 
\prod_{i\in[h]}\E X_{t_i}^{s+2},
\qquad
\sum_{j\in [m]\setminus\{t_1,\dots, t_h\}}\E X_j^2
\le\sum_{j\in[m]}\E X_j^2
=
m\E X_{i*}^2.
\end{align*}
We obtain that 
\begin{align*}
\PP\{{\cal F}({\tilde S}, \pi, {\tilde t})\cap {\cal I}(S,U,[m])
\cap
{\cal I}_U\}
&
\le
\frac{r^h}{n^{s+h+1}}
\left(
r\frac{h}{2}+\frac{m}{n}\E X_{i*}^2
\right)
\\
&
\times
e^{-r\left(\kappa\frac{m}{n}-o(1)\right)}
\prod_{i\in[h]}\E X_{t_i}^{s+2}.
\end{align*}
Note that the  quantity on the right
does not depend on the partition 
${\tilde S}\in {\mathbb S}_h$
and permutation $\pi$. 
Next, using   (\ref{Fact1})
we bound the sum of products 
\begin{align*}
h!\sum_{{\tilde t}\in T_h}\prod_{i\in[h]}\E X_{t_i}^{s+2}
\le 
m^h\E X_{i*}^{s+2}.
\end{align*}
Combining the latter two inequalities we upper bound the sum on the right of 
 (\ref{2025-08-01+r}). We have
\begin{align}
\nonumber
p_{s,r}
&
\le
\sum_{h=1}^s
|{\mathbb S}_h|
\frac{r^hm^h}{n^{s+h+1}}\E X_{i*}^{s+2}
\left(
r\frac{h}{2}+\frac{m}{n}\E X_{i*}^2
\right)
e^{-r\left(\kappa\frac{m}{n}-o(1)\right)}
\\
\nonumber
&
\le
c'
\frac{1}{n^{s+1}}\sum_{h=1}^s\left(r^{h+1}\frac{m^h}{n^h}
+
r^h\frac{m^{h+1}}{n^{h+1}}\right)
e^{-r\left(\kappa\frac{m}{n}-o(1)\right)}
\\
\label{2025-08-13+3+r}
&
\le
c''r^{s+1}
\frac{1}{n^{s+1}}
\left(\frac{m}{n}\right)^{s+1}
e^{-r\left(\kappa\frac{m}{n}-o(1)\right)}.
\end{align}

Here and below $c', c'', c'''$ denote constants that do not depend on $m,n,r$.
In the second inequality we uper bounded   the number of partitions
$|{\mathbb S}_h|$
 (Stirling's number of the second kind)     by  a constant (depending on $s$, but not depending on $n,m,r$). 

We invoke bound (\ref{2025-08-13+3+r}) in the formula for $S_1$ (see (\ref{2025-07-24+r}) and below). We have
\begin{align*}
S_1
\le
\sum_{s=1}^{k}\frac{n^s}{s!}
\sum_{2\le r\le n^{\beta}}
\frac{n^r}{r!}p_{s,r}
\le
c'''
\sum_{s=1}^{k}
\frac{1}{s!}\frac{m^{s+1}}{n^{s+2}}
\sum_{2\le r\le n^{\beta}}
\frac{r^{s+1}}{r!}
e^{r\left(\lambda(0)+o(1)\right)}.
\end{align*}
 Our assumptions $m=O(n\ln n)$ and $\lambda(0)\to-\infty$ implies 
 $S_1=o(1)$ as $n,m\to+\infty$.
 
 \smallskip
 
 {\it Proof of} $S_2=o(1)$. For
 $\alpha_n>{\tilde \alpha}$
 Lemma \ref{q-bounds} (iii) shows
 \begin{align*}
 p^*_{s,r}
 \le e^{-{\tilde \alpha} r\frac{m}{n}},
 \qquad
 1\le r\le n/2.
 \end{align*}
Combining this inequality with (\ref{2026-01-30}) and 
$\binom{n}{s}\le  n^s$ we obtain
for $n^{\beta}\le r\le (n-s)/2$
\begin{align*}
\binom{n}{s}
\binom{n}{r} 
p^*_{s,r}
\le
e^{-\alpha r\frac{m}{n}+2r+(1-\beta)r\ln n+s\ln n}.
\end{align*}
In view of 
 identities 
$\frac{m}{n}=\frac{\ln n-\lambda(0)}{\kappa}$ and 
$1-\beta=\frac{{\tilde \alpha}}{2\kappa}$ we write the quantity on the right in the form
\begin{align*}
e^{-r\left(\frac{{\tilde \alpha}}{2\kappa}\ln n-\frac{{\tilde \alpha}}{\kappa}\lambda(0) -2-\frac{s}{r}\ln n\right)}.
\end{align*}
Finally, since $\lambda(0)\to-\infty$ and $\frac{s}{r}\ln n\le sn^{-\beta}\ln n=o(1)$, we conclude that
$S_2=o(1)$.
\end{proof}

\begin{proof}[Proof of Lemma \ref{L1}]
We upper bound $\PP\{{\cal A}\}$ using 
the union bound,
\begin{align*}
\PP\{{\cal A}\}
&
\le 
\sum_{\{i,j,r\}\subset [m]}
\sum_{\{u,v\}\subset {\cal V}}
\PP\{d_\ell(u)>0, d_\ell(v)>0, \forall \ell\in\{i,j,r\}\}
\\
&
=
\binom{n}{2}
\sum_{\{i,j,r\}\subset [m]}
\frac{\E(X_i)_2}{(n)_2}
\frac{\E(X_j)_2}{(n)_2}
\frac{\E(X_r)_2}{(n)_2}
\\
&
\le
\binom{n}{2}
\left(\frac{m}{(n)_2}\right)^3
\frac{1}{3!}
\left(\E (X_{i_*})_2\right)^3
=o(1).
\end{align*}
In the second inequality we used (\ref{Fact1}) for $b=3$.

Let us show that $\PP\{N'_k\ge 1\}=o(1)$.
On the event ${\bar {\cal A}}$ (complement event to 
${\cal A}$) we have
\begin{align*}
N'_k
=
\sum_{v\in {\cal V}}\sum_{u\in{\cal V}\setminus\{v\}}
{\mathbb  I}_{\{d'(v)= k\}} 
{\mathbb  I}_{\{d'(u,v)=2\}}=:N''_k.
\end{align*}
 Next we evaluate $\E N''_k$. At this point we need  some more notation. 
 For 
 $u,v\in{\cal V}$ and $\{i,j\}\subset [m]$ let $A_{u,v}(i,j)$ denote the event that $d_\ell(u)>0,d_\ell(v)>0$ for each 
 $\ell\in\{i,j\}$. Furthermore, for $B\subset [m]$
 let $A_v(B)$ denote the event that 
 $d_\ell(v)>0$ for 
 each $\ell\in B$. Let $A^*_v(B)$ denote the event that $d_\ell(v)=0$ for each $\ell\in[m]\setminus B$.
 We have, by  symmetry,
\begin{align}
\nonumber
\E N''_k
&
=
(n)_2\PP\{d'(u,v)=2, d'(v)=k\}.
\\
\nonumber
&
=
(n)_2
\sum_{\{i,j\}\subset[m]}
\PP\{A_{u,v}(i,j),  d'(v)=k\}
\\
\label{2026-01-14}
&
=
(n)_2
\sum_{\{i,j\}\subset[m]}
\sum_{\substack{B\subset [m]\setminus\{i,j\}\\
|B|=k-2}}
\PP\{A_{u,v}(i,j), A_v(B), 
A^*_v(B\cup\{i,j\})\}.
\end{align}
By the independence of $G_1,\dots, G_m$, we have
\begin{align*}
\PP\{A_{u,v}(i,j), A_v(B), 
A^*_v(B\cup\{i,j\})\}
=
\PP\{A_{u,v}(i,j)\}
\PP\{A_v(B)\}
\PP\{A^*_v(B\cup\{i,j\})\}.
\end{align*}
Furthermore, we have
\begin{align*}
&
\PP\{A_{u,v}(i,j)\}
=
\frac{\E(X_i)_2}{(n)_2}
\frac{\E(X_j)_2}{(n)_2},
\qquad
\PP\{A_v(B)\}
=
\prod_{\ell\in B}\frac{x_\ell}{n},
\\
&
\PP\{A^*_v(B\cup\{i,j\})\}
=
\prod_{\ell\in[m]\setminus(B\cup
\{i,j\})}
\left(1-\frac{x_\ell}{n}\right)
=
e^{-\frac{m}{n}\kappa}(1+o(1)).
\end{align*}
The very last approximation  follows by
(\ref{T}) 
and 
(\ref{2026-01-10+6}). It is important to note that the bound $o(1)$
holds uniformly in
$\{i,j\}$ and $B$ with $|B|=k-2$.
Invoking these identities in (\ref{2026-01-14}) we obtain that
\begin{align*}
\E N''_k
\le 
(n)_2
e^{-\frac{m}{n}\kappa}(1+o(1))
\Biggl(
\sum_{\{i,j\}\subset[m]}
\frac{\E(X_i)_2}{(n)_2}
\frac{\E(X_j)_2}{(n)_2}
\Biggr)
\Biggl(
\sum_{\substack{B\subset [m]\setminus\{i,j\}\\
|B|=k-2}}
\prod_{\ell\in B}\frac{x_\ell}{n}
\Biggr)
\end{align*}
Let $I_1$ and $I_2$ denote the quantities in the second-to-last and last parentheses on the right, respectively. Using (\ref{Fact1}) we upperbound  
\begin{align*}
I_1
&
\le 
\frac{1}{2}
\left(\frac{m}{(n)_2}\right)^2
(\E (X_{i_*})_2)^2
=
O\left(\frac{m^2}{n^4}\right),
\\
I_2
&
\le 
\sum_{\substack{B\subset [m]\\
|B|=k-2}}
\prod_{\ell\in B}\frac{x_\ell}{n}
\le
\frac{1}{(k-2)!}
\left(\frac{m}{n}\right)^{k-2}
(\E X_{i_*})^{k-2}
=
O\left(\frac{m^{k-2}}{n^{k-2}}\right).
\end{align*}
We conclude that 
$\E N''_k=O\left(\frac{m^k}{n^k}e^{-\frac{m}{n}\kappa}\right)$. Note that $\frac{m^k}{n^k}e^{-\frac{m}{n}\kappa}=o(1)$ for $\frac{m}{n}\to+\infty$.
Hence $\E N''_k=o(1)$. 
Now Markov's 
inequality yields 
$\PP\{N''_k\ge 1\}\le \E N''_k=o(1)$.
 Finally, 
\begin{align*}
\PP\{N'_k\ge 1\}
=
\PP\{N'_k\ge 1,{\cal A}\}
+
\PP\{N'_k\ge 1, {\bar{\cal A}}\}
\le
\PP\{{\cal A}\}
+
\PP\{N''_k\ge 1\}
=o(1).
\end{align*}

\medskip

Now we evaluate  $\E N_k$ and
$\E N_{*k}$. To  this aim we use approximations
(\ref{2025-09-04}) and
(\ref{2025-06-21_6+r+r}) shown in Lemma \ref{AB+} below. Fix $v\in{\cal V}$.  We have, by symmetry,
\begin{align*}
\E N_k
&
=
n\PP\{d'(v)=k\}
=
n 
\frac{\kappa^k}{k!}\left(\frac{m}{n}\right)^k 
e^{-\kappa\frac{m}{n}}
 (1+o(1))
\asymp
e^{\ln n+k\ln\frac{m}{n}-\kappa\frac{m}{n}},
\\
\E N_{*k}
&
=
n\PP\{d_*'(v)=d'(v)=k\}
=
n 
\frac{\kappa_a^k}{k!}\left(\frac{m}{n}\right)^k e^{-\kappa\frac{m}{n}}
 (1+o(1))
\asymp
e^{\ln n+k\ln\frac{m}{n}-\kappa\frac{m}{n}}.
\end{align*}
Now $\lambda_n(k)\to+\infty$ (respectively $\lambda_n(k)\to-\infty$)
implies $\E N_k\to+\infty$ and $\E N_{*k}\to+\infty$ 
(respectively,  $\E N_k\to 0$ and $\E N_{*k}\to 0$).

Next we  show that $\lambda_n(k)\to+\infty$ implies $N_{*k}=(1+o(1))\E N_{*k}$. To this aim we evaluate the variance of $N_{*k}$ and apply Chebyshev's inequality.
We calculate the expected value 
\begin{align*}
\E \binom{N_{*k}}{2}
&
=
\E
\left(
\sum_{\{u,v\}\subset {\cal V}}
{\mathbb I}_{\{d_*'(u)=d'(u)=k\}}{\mathbb I}_{\{d_*'(v)=d'(v)=k\}}\right)
\\
&
=
\binom{n}{2}\PP\{d_*'(u)=d'(u)=k, d_*'(v)=d'(v)=k\}
\\
&
=
\binom{n}{2}
\frac{\kappa_a^{2k}}{(k!)^2}\left(\frac{m}{n}\right)^{2k} e^{-2\kappa\frac{m}{n}}
 (1+o(1)). 
\end{align*}
In the last step we invoked (\ref{2025-6-21_9+r}). 
Combining expressions for $\E N_{*k}$ and 
$\E \binom{N_{*k}}{2}$ above we evaluate the variance
\begin{align*}
\Var N_{*k}
=
\E N_{*k}^2
-
(\E N_{*k})^2
=
2\E \binom{N_{*k}}{2}
+
\E N_{*k}
-
(\E N_{*k})^2
=
o((\E N_{*k})^2).
\end{align*}
For $\E N_{*k}\to+\infty$ the bound
$\Var N_{*k}=o((\E N_{*k})^2)$ implies $N_{*k}=(1+o_P(1))\E N_{*k}$, by Chebyshev's inequality.
\end{proof}

\begin{proof}[Proof of Corollary \ref{col2}]  
We note that
(\ref{a}),
(\ref{condition_X^2}) imply $\kappa\asymp 1$. 
Now 
(\ref{2026-01-14+2}) 
implies
$m\asymp n\ln n$.

Proof of (i). Left inequality of (\ref{2026-01-14+2}) 
implies 
$\lambda_n(r)\to-\infty$
for $r=0, 1,\dots k-1$.
 Now relation $\PP\{N_r=0\}=1-o(1)$ follows from Lemma \ref{L1}.

Proof of (ii).  
Right inequality of (\ref{2026-01-14+2}) 
implies 
$\lambda_n(r)\to+\infty$
for any $r\ge k$.
Now  Lemma \ref{L1} implies 
$N_{*r}=(1+o_P(1))\E N_{*r}$ and $\E N_{*r}\to+\infty$. Hence for any $A>0$ we have $\PP\{N_{*r}>A\}=1-o(1)$. Furthermore, Lemma \ref{L1} shows
$\PP\{N'_r=0\}=1-o(1)$.
Finally, event $N'_r=0$ 
implies that each $v$ with $d'(v)=r$ is a 
center of an $r$-blossom.

Proof of (iii). 
By (i) whp there is no vertex $v$ with $d'(v)<k$. 
By (ii) there is a large number ($=N_{*k}$) of vertices $v$ with $d(v)=\sum_{i\in[m]}d_i(v)=ka$. 
We claim that there is no vertex $w$ with $d(w)<ka$.  Indeed, for
$w$ with $d'(w)\in\{k,k+1,\dots, 2k\}$ we use the fact (shown in (ii)) that $w$ is the center of a $d'(w)$-blossom to bound the degree $d(w)$ from below
$d(w)\ge d'(w) a \ge ka$. For $w$ with $d'(w)>2k$ we use the fact (shown in Lemma \ref{L1}) that whp
$d'(w,u)\le 2$, for any  $u\in {\cal V}\setminus\{w\}$.
In particular, $d'(w,u)\le 2$ for each $u\in {\cal N}_w$, where ${\cal N}_w$ denotes the set of neighbours of $w$ in
$G_{[n,m]}$.  Now the chain of inequalities
\begin{align*}
2d(w)
=
\sum_{u\in{\cal N}_w}2
\ge
\sum_{u\in{\cal N}_w}d'(w,u)
= 
\sum_{i\in[m]}d_i(v)
\ge 
d'(w) a
\ge 2ka
\end{align*}
implies $d(w)\ge ka$. 
\end{proof}

\section{Auxiliary results}

\subsection{Degree probabilities}

First, we introduce some shorthand notation and make several observations.
We denote
\begin{align*}
x_i
&
=
\E X_i,
\qquad
x_{a,i}
=
\E X_{a,i},
\qquad
z_{a,i}=\E(X_{a,i})_2,
\qquad
z_i=\E(X_i)_2,
\\
T
&
=
\PP\{d'(1)=0\},
\qquad
H
=
\PP\{d'(1)=d'(2)=0\}
\end{align*}
and observe that $x_1+\cdots+x_m=m\E X_{i_*}=m\kappa$ and
$x_{a,1}+\cdots+x_{a,m}=m\E X_{i_*}=m\kappa_a$. 
Using
$1-z=e^{\ln(1-z)}=e^{-z+O(z^2)}$ for $z=o(1)$ and  (\ref{condition_X^2})  we approximate
for $m=o(n^2)$
\begin{align}
\label{T}
T
&
=
\prod_{i\in[m]}
\PP\{d_i(1)=0\}
=
\prod_{i\in[m]}
\left(1-\frac{x_i}{n}\right)
=
e^{-\frac{m}{n}\kappa}(1+o(1)),
\\
\nonumber
H
&
=
\prod_{i\in[m]}
\PP\{d_i(1)=d_i(2)=0\}
=
\prod_{i\in[m]}
\E
\left(
\left(1-\frac{X_i}{n}\right)
\left(1-\frac{X_i}{n-1}\right)
\right)
\\
\label{H}
&
=
\prod_{i\in[m]}
\left(
1-2\frac{x_i}{n}+\frac{z_i}{(n)_2}\right)
=
e^{-2\frac{m}{n}\kappa}(1+o(1)).
\end{align}
Furthermore,  for $m=O(n\ln n)$  we have 
\begin{align}
\label{2026-01-10+6}
\max_{i\in [m]}x_i^2=O(n\ln n),
\qquad 
\max_{i\in [m]}z_i=O(n\ln n).
\end{align}
These bounds follow from  (\ref{condition_X^2}) via the chain of inequalities
\begin{align*}
\max_{i\in[m]} (\E X_i)^2
\le 
\max_{i\in[m]} \E X_i^2
\le 
\sum_{i\in [m]}\E X_{i}^2
=
m\E X_{i_*}^2 =O(m)=O(n\ln n).
\end{align*}
Next, using  (\ref{2026-01-10+6}) 
we approximate
 for any fixed integer  $b\ge 1$ 
as $n,m\to+\infty$
\begin{align}
\label{2026-01-10+1}
&
\max_{B\subset [m],\, |B|\le b} \prod_{i\in B}\left(1-\frac{x_i}{n}\right)
=
1-O\left(\frac{\sqrt{\ln n}}{\sqrt{n}}\right),
\\
\label{2026-01-10+2}
&
\max_{B\subset [m],\, |B|\le b} \prod_{i\in B}\left(1-2\frac{x_i}{n}+\frac{z_i}{(n)_2}\right)
=
1-O\left(\frac{\sqrt{\ln n}}{\sqrt{n}}\right).
\end{align}

We will use the following simple inequality.
Let $a_1,\dots, a_m$ be non-negative real numbers. Denote ${\bar a}=m^{-1}(a_1+\cdots+a_m)$.
For an integer $b\ge 2$ we have
\begin{align}
\label{Fact1}
b!\sum_{B\in\binom{[m]}{b}}\prod_{i\in B}a_i
=
(a_1+\cdots+a_m)^b
-
R
\le 
(a_1+\cdots+a_m)^b,
\end{align}
where
\begin{align*}
0
\le 
R
\le
\frac{(b)_2}{2}
(a_1^2+\cdots+a_m^2)
(a_1+\cdots+a_m)^{b-2}.
\end{align*}
Proof of (\ref{Fact1}). 
By the multinomial formula, we have
\begin{align*}
(a_1+\cdots+a_m)^b
=
b!\sum_{B\in\binom{[m]}{b}}\prod_{i\in B}a_i
+R,\
\end{align*}
where
\begin{align*}
R
&
=
\sum_{i=1}^m
a_i^2
\sum_{p_1+\dots+p_m=b-2}
\frac{b!}{p_1!\cdots p_{i-1}!(p_i+2)!p_{i+1}!\cdots p_m!}
\prod_{i=1}^m
a_i^{p_i}
\\
&
=
\sum_{i=1}^m
a_i^2
\sum_{p_1+\dots+p_m=b-2}
\frac{(b)_2}{(p_i+2)_2}
\frac{(b-2)!}{p_1!\cdots p_{i-1}!p_i!p_{i+1}!\cdots p_m!}
\prod_{i=1}^m
a_i^{p_i}
\\
&
\le 
\frac{(b)_2}{2}
\sum_{i=1}^m
a_i^2
\sum_{p_1+\dots+p_m=b-2}
\frac{(b-2)!}{p_1!\cdots p_{i-1}!p_i!p_{i+1}!\cdots p_m!}
\prod_{i=1}^m
a_i^{p_i}
\\
&
=
\frac{(b)_2}{2}
(a_1^2+\cdots+a_m^2)
(a_1+\dots+a_m)^{b-2}.
\end{align*}

\begin{lem}\label{AB+} 
Let $a\ge 1$ and $k\ge 0$ be integers.
Let $m,n\to+\infty$.
 Assume that $m=m(n)=o(n\ln^2n)$ and $n\ln n=O(m)$.  Assume that (\ref{condition_X^2}) holds. Then
 \begin{align}
  \label{2025-09-04}
 \PP\{d'(1)=k\}
 &
 =
 \frac{\kappa^k}{k!}\left(\frac{m}{n}\right)^k e^{-\kappa\frac{m}{n}}
 (1+o(1)),
 \\
 \label{2025-06-21_6+r+r}
 \PP\{d_*'(1)=d'(1)=k\}
 &
 =
 \frac{\kappa_a^k}{k!}
 \left(\frac{m}{n}\right)^k
  e^{-\kappa\frac{m}{n}}
 (1+o(1)).
 \end{align}
 Assume, in addition, that $\liminf\kappa_a>0$.
 For $k\ge 0$ we have 
 \begin{align}
  \label{2025-6-21_9+r}
 \PP\{d_*'(1)=d'(1)=k,\, d_*'(2)=d'(2)=k\}
& =\frac{\kappa_a^{2k}}{(k!)^2}
\left(\frac{m}{n}\right)^{2k}
 e^{-2\kappa\frac{m}{n}}
 (1+o(1)).
 \end{align}
\end{lem}
\begin{proof}[ Proof of Lemma \ref{AB+}]
For $k=0$ (\ref{2025-09-04}), (\ref{2025-06-21_6+r+r}) follow from (\ref{T}) via identities
\begin{align*}
\PP\{d'(1)=0\}
=
\PP\{d'(1)=d'_*(1)=0\}
= 
T.
\end{align*}
Similarly,  (\ref{2025-6-21_9+r}) follows from (\ref{H})
via identity
\begin{align*}
\PP\{d_*'(1)=d'(1)=0,\, d_*'(2)=d'(2)=0\}
=
\PP\{d'(1)=d'(2)=0\}=H.
\end{align*}

For the rest of the proof we assume that $k\ge 1$.

Proof of (\ref{2025-09-04}).
 We apply the total probability formula and use the independence of 
 $G_1,\dots, G_m$:
 \begin{align*}
 \PP\{d'(1)=k\}
&
=
\sum_{B\in\binom{[m]}{k}}
\PP
\left\{
d_i(1)>0,\, i\in B
\
\ 
{\text{and}}
\
\
d_j(1)=0, \, j \in[m]\setminus B
\right\}
\\
\nonumber
&
=
\sum_{B\in\binom{[m]}{k}}
\left(
\prod_{i\in B}
\PP\{d_i(1)>0\}
\right)
\left(
\prod_{i\in[m]\setminus B}
\PP\{d_j(1)=0\}
\right)
\\
\nonumber
&=
\sum_{B\in\binom{[m]}{k}}
\left(
\prod_{i\in B}
\frac{x_{i}}{n}
\right)
\left(
\prod_{j\in[m]\setminus B}
\left(
1
-
\frac{x_j}{n}
\right)
\right)
\\
&
=
T
\sum_{B\in\binom{[m]}{k}}
\prod_{i\in B}\frac{x_{i}}{n}
\left( 1-\frac{x_i}{n}\right)^{-1}.
\end{align*}
In view of (\ref{2026-01-10+1}) we 
have $\prod_{i\in B}
\left( 1-\frac{x_i}{n}\right)^{-1}=
1+o(1)$ uniformly 
in  $B$.
Hence,
 \begin{align*}
 \PP\{d'(1)=k\}
=
(1+o(1))
T\frac{1}{n^k}
\sum_{B\in\binom{[m]}{k}}
\prod_{i\in B}x_{i}
\end{align*}
Finally, 
invoking 
 (\ref{T}) and approximation 
$\sum_{B\in\binom{[m]}{k}}
\prod_{i\in B}x_{i}
=
\frac{m^k}{k!}\kappa^k+O(m^{k-1})$, which follows by 
(\ref{Fact1}),  we obtain 
 (\ref{2025-09-04}).

Proof of (\ref{2025-06-21_6+r+r}). We proceed similarly as in the proof of (\ref{2025-09-04})
above. We have
\begin{align*}
\PP\{d'_*(1)=d'(1)=k\}
&
=
\sum_{B\in\binom{[m]}{k}}
\PP\{d_i(1)=a,\, i\in B
\
\
{\text{and}}
\
\
d_j(1)=0,
\, j\in [m]\setminus B\}
\\
\nonumber
&
=
\sum_{B\in\binom{[m]}{k}}
\left(
\prod_{i\in B}
\PP\{d_i(1)=a\}
\right)
\left(
\prod_{i\in[m]\setminus B}
\PP\{d_j(1)=0\}
\right)
\\
\nonumber
&=
\sum_{B\in\binom{[m]}{k}}
\left(
\prod_{i\in B}
\frac{x_{a,i}}{n}
\right)
\left(
\prod_{j\in[m]\setminus B}
\left(
1
-
\frac{x_j}{n}
\right)
\right)
\\
\nonumber
&
=
(1+o(1))
T\frac{1}{n^k}
\sum_{B\in\binom{[m]}{k}}
\prod_{i\in B}x_{a,i}
\\
\nonumber
&
=
 (1+o(1))\frac{m^k}{n^k}\frac{\kappa_a^k}{k!}
 e^{-\frac{m}{n}\kappa}.
\end{align*}

Proof of (\ref{2025-6-21_9+r}). Denote for short
$p=\PP\{d'_*(1)=d'(1)=d'_*(2)=d'(2)=k\}$.
Given mutually disjoint sets $B_1,B_2,B_3\subset [m]$, let 
$I_{B_1,B_2,B_3}$
denote the event that
\begin{align*}
&
d_i(1)=d_i(2)=a
\
 \forall i\in B_1,
\quad
d_j(1)=a, 
\
d_j(2)=0
\
 \forall j\in B_2,
\\
&
d_h(1)=0, 
\ 
d_h(2)=a
\
 \forall h\in B_3,
\quad
d_\ell(1)=d_\ell(2)=0
\
\forall \ell\in [m]\setminus (B_1\cup B_2\cup B_3).
 \end{align*}
By the total probability formula we have 
\begin{align*}
p
=
\sum_{s=0}^k{\bar p}_s,
\qquad
{\bar p}_s
=
\sum_
{\substack{B_1\subset [m]\\ |B_1|=s}}
\
\sum_{\substack{B_2\subset [m]\setminus B_1\\
|B_2|=k-s}}
\
\sum_{\substack{B_3\subset [m]\setminus (B_1\cup B_2)\\
|B_3|=k-s}}
\PP\{I_{B_1,B_2,B_3}\}.
\end{align*} 
Furthermore, by the independence of $G_1,\dots, G_m$,
we factorize the probability 
\begin{align}
\nonumber
\PP\{I_{B_1,B_2,B_3}\}
&=
\left(
\prod_{i\in B_1}
p_i
\right)
\times
\left(
\prod_{j\in B_2}p_j(1,2)
\right)
\times
\left(
\prod_{h\in B_3}p_h(2,1)
\right)
\times
\prod_{\ell\in [m]\setminus(B_1\cup B_2\cup B_3)}
q_\ell
\\
\label{2026-01-10+4}
&
=:
P_1(B_1)
\times
P_2(B_2)
\times
P_3(B_3)
\times
Q(B_1,B_2,B_3).
\end{align}
Here we denote
 \begin{align*}
&
p_i=\PP\{d_i(1)=d_i(2)=a\},
\qquad
q_i=\PP\{d_i(1)=d_i(2)=0\},
\\
&
p_i(1,2)=\PP\{d_i(1)=a, \,d_i(2)=0\},
\qquad
p_i(2,1)=\PP\{d_i(1)=0,\, d_i(2)=a\}.
\end{align*}
A calculation shows that
\begin{align*}
p_i
&
=
\E \frac{(X_{a,i})_2}{(n)_2}=\frac{z_{a,i}}{(n)_2},
\qquad
q_i
=
\E\frac{(n-X_i)_2}{(n)_2}
=1-2\frac{x_i}{n}+\frac{z_i}{(n)_2},
\\
p_i(1,2)
&
=
p_i(2,1)
=
\E\frac{X_{a,i}(n-X_i)}{(n)_2}
=
\frac{x_{a,i}}{n-1}-\frac{\E (X_iX_{a,i})}{(n)_2}
\le 
\frac{x_{a,i}}{n-1}.
\end{align*}
We also note that in view of (\ref{2026-01-10+2})  the last term on the right of (\ref{2026-01-10+4})
\begin{align}
\label{2026-01-12+5}
Q(B_1,B_2,B_3)
=
H\prod_{\ell\in B_1\cup B_2\cup B_3}
\left(
1-2\frac{x_i}{n}+\frac{z_i}{(n)_2}
\right)^{-1}
=
H
(1+o(1)),
\end{align}
where the bound $o(1)$ holds
uniformly over $B_1,B_2,B_3$ satisfying $|B_1\cup B_2\cup B_3|\le 2k$.

In the remaining part of the proof we show that
\begin{align}
\label{2026-01-10+5}
{\bar p}_0
=
\frac{\kappa_a^{2k}}{(k!)^2}
\left(
\frac{m}{n}
\right)^{2k}
H
\left(
1
+
o(1)
\right)
\qquad
{\text{and}}
\qquad
\sum_{s=1}^k {\bar p}_s
=
O\left(H\frac{m^{2k-1}}{n^{2k}}\right).
\end{align}
These relations combined with (\ref{H}) imply
(\ref{2025-6-21_9+r}).

Before the proof  of
(\ref{2026-01-10+5}) 
we introduce some notation.
For an integer $b\ge 1$ we denote
\begin{align*}
S_{1,b}
=\sum_
{\substack{B_1\subset [m],\\ |B_1|=b}}
\prod_{i\in B_1}
p_i,
\quad
\
S_{2,b}=
\sum_{\substack{B_2\subset [m]\\
|B_2|=b}}
\prod_{j\in B_2}p_j(1,2),
\quad
\
S_{3,b}=
\sum_{\substack{B_3\subset [m]\\
|B_3|=b}}
\prod_{h\in B_3}p_h(2,1).
\end{align*}
Note that $S_{2,b}=S_{3,b}$.
Next we establish several useful facts about the sums $S_{1,b}$ and $S_{2,b}$.

Using
$p_i=\frac{z_{a,i}}{(n)_2}$ 
and
$p_i(1,2)
=
p_i(2,1)
\le 
\frac{x_{a,i}}{n-1}$, 
and (\ref{Fact1}) we upperbound
\begin{align}
\label{2026-01-12+2}
S_{1,b}
\le 
\frac{1}{b!}
\left(\frac{m}{(n)_2}\right)^b
\left(
\E(X_{a,i_*})_2
\right)^b,
\qquad
S_{r,b}
\le 
\frac{1}{b!}
\left(\frac{m}{n-1}\right)^b
\left(
\E X_{a,i_*}
\right)^b,
\quad
r=2,3.
\end{align}
 Moreover, 
the second relation can be upgraded
to the approximate identity (recall that $\E X_{a,i_*}=\kappa_a$)
\begin{align}
\label{2026-01-12+3}
S_{2,b}
=
 \frac{\kappa_a^b}{b!}\left(\frac{m}{n}\right)^b
 \left(1+O(n^{-1})\right).
\end{align}
Let us show (\ref{2026-01-12+3}). 
Our conditions $\liminf_n\kappa_a>0$ and (\ref{condition_X^2}) imply 
$\kappa_a^{-1}\E(X_{i_*}X_{a,i_*})=O(1)$.
Now, for $b=1$ we have 
\begin{align*}
S_{2,1}
=
\sum_{j\in [m]}p_j(1,2)
=m\left(\frac{\kappa_a}{n-1}-\frac{\E (X_{i_*}X_{a,i_*})}{(n)_2}\right)
=
\kappa_a
\frac{m}{n}
\left(
1
+
O(n^{-1})
\right).
\end{align*}
For $b\ge 2$ 
 we combine 
 the upper bound 
$S_{2,b}
\le
\frac{\kappa_a^b}{b!}
\left(\frac{m}{n}\right)^b
\left(1+O\left(\frac{1}{n}\right)\right)$,
which 
follows from the the second inequality of 
(\ref{2026-01-12+2}), with a matching lowe bound. We show
the
lower bound using 
(\ref{Fact1}), we have
\begin{align*}
S_{2,b}
\ge \frac{S}{b!}-\frac{R}{2(b-2)!},
\end{align*}
where
\begin{align*}
S=\left(
\sum_{i=1}^mp_i(2,1)\right)^b
\qquad
{\text{and}}
\qquad
R=\left(
\sum_{i=1}^mp_i(2,1)\right)^{b-2}
\sum_{i=1}^mp_i^2(2,1).
\end{align*}
Invoking 
$p_i(1,2)= \frac{x_{a,i}}{n-1}-\frac{\E (X_iX_{a,i})}{(n)_2}$
we write $S$ in the form 
\begin{align*}
S
&
=
\left(\frac{m}{n-1}\right)^b
\left(\kappa_a-\frac{\E(X_{i_*}X_{a,i_*})}{n}\right)^b
=
\left(\frac{m}{n}\right)^b
\kappa_a^b 
\left(1+O(n^{-1})\right).
\end{align*}
Furthermore, using  
$p_i(1,2)\le \frac{x_{a,i}}{n-1}$ 
we upper bound
\begin{align*}
R
\le 
\left(\frac{m}{n-1}\kappa_a\right)^{b-2}
\frac{m}{(n-1)^2}\E X^2_{a,i_*}
=
O\left(\kappa_a^{b-2}\E X_{a,i_*}^2
\frac{m^{b-1}}{n^b}\right)
=
O\left(\kappa_a^b
\frac{m^{b-1}}{n^b}\right).
\end{align*}
In the last step we used the bound
$\kappa_a^{-2}\E X_{a,i_*}^2=O(1)$, which follows from 
our conditions
(\ref{condition_X^2})  and
$\liminf_n\kappa_a>0$.
We arrive to the bound 
$S_{2,b}
\ge 
\frac{\kappa_a^b}{b!}
\left(\frac{m}{n}\right)^b
\left(
1
-
O(n^{-1})
\right)$
thus completing the proof of
(\ref{2026-01-12+3}).

Let us show (\ref{2026-01-10+5}) for $k=1$. 
We have
\begin{align*}
{\bar p}_1
=
\sum_{i\in [m]}p_iQ(\{i\},\emptyset,\emptyset)
=
\sum_{i\in [m]}p_iH(1+o(1))
=
S_{1,1}H(1+o(1))
=
O\left(\frac{m}{n^2}H\right).
\end{align*}
Here we approximated $Q(\{i\},\emptyset,\emptyset)
=
H(1+o(1))$ by  (\ref{2026-01-12+5}) 
and then bounded $S_{1,1}$ by  (\ref{2026-01-12+2}).
We similarly approximate
\begin{align*}
{\bar p}_0
&
=
\sum_{j\in[m]}p_j(1,2)\sum_{h\in[m]\setminus\{j\}}p_h(2,1)Q(\emptyset,\{j\},\{h\})
\\
&
=
\sum_{j\in[m]}p_j(1,2)\sum_{h\in[m]\setminus\{j\}}p_h(2,1)H(1+o(1))
\\
&
=
\left(
S_{2,1}
S_{3,1}
-
\sum_{j\in[m]}p_j(1,2)p_j(2,1)
\right)
H(1+o(1)).
\end{align*}
Invoking the approximation $S_{2,1}S_{3,1}=S_{2,1}^2
=\kappa_a^2(m/n)^2(1+o(1))$, see (\ref{2026-01-12+3}), and bound
\begin{align*}
\sum_{j\in[m]}p_j(1,2)p_j(2,1)
\le \sum_{j\in[m]}\frac{x_{a,j}^2}{(n-1)^2}
=
m\frac{\E X_{a,i_*}^2}{(n-1)^2}
=
O\left(\frac{m}{n^2}\right)
\end{align*}
we obtain 
${\bar p}_0
=
\kappa_a^2(m/n)^2H(1+o(1))+O\left(\frac{m}{n^2}H\right)$
thus arriving to  (\ref{2026-01-10+5}).

Now  we show (\ref{2026-01-10+5}) for $k\ge 2$. 
Let us upper bound ${\bar p}_s$ for $s\ge 1$.
By increasing the range of summation we upper bound
\begin{align*}
{\bar p}_s
&
\le
\sum_
{\substack{B_1\subset [m]\\ |B_1|=s}}
\
\sum_{\substack{B_2\subset [m]\\
|B_2|=k-s}}
\
\sum_{\substack{B_3\subset [m]\\
|B_3|=k-s}}
P_1(B_1)P_2(B_2)P_3(B_3)Q(B_1,B_2,B_3)
\\
&
=
S_{1,s}S_{2,k-s}S_{3,k-s}H(
1+o(1)).
\end{align*}
In the last step we invoked (\ref{2026-01-12+5}).
Now (\ref{2026-01-12+2}) implies
${\bar p}_s=O\left(H\frac{m^{2k-s}}{n^{2k}}\right)$. 
Consequently, we obtain 
$\sum_{s=1}^k{\bar p}_s
=O\left(H\frac{m^{2k-1}}{n^{2k}}\right)$.

Let us show the first relation of (\ref{2026-01-10+5}).
We write ${\bar p}_0$ in the form
${\bar p}_0
=
H{\tilde p}_0+{\tilde R}$,
where
\begin{align*}
{\tilde p}_0
&
=
\sum_{\substack{B_2\subset [m]\\
|B_2|=k}}
P_2(B_2)
\sum_{\substack{B_3\subset [m]\setminus B_2\\
|B_3|=k}}
P_3(B_3),
\\
{\tilde R}
&
=
\sum_{\substack{B_2\subset [m]\\
|B_2|=k}}
P_2(B_2)
\sum_{\substack{B_3\subset [m]\setminus B_2\\
|B_3|=k}}
P_3(B_3) (Q(\emptyset, B_2,B_3)-H).
\end{align*}
Using  (\ref{2026-01-12+5}) we estimate
${\tilde R}=o\left({\tilde p}_0H
\right)$.
Furthermore, using
(\ref{2026-01-12+2}) we estimate
\begin{align*}
{\tilde p}_0
\le 
S_{2,k}S_{3,k}
=
O
\left(\frac{m^{2k}}{n^{2k}}\right).
\end{align*}
We conclude that 
${\tilde R}
=
o\left(H\frac{m^{2k}}{n^{2k}}
\right)$.
Let us evaluate ${\tilde p}_0$.
We have
\begin{align}
\label{2026-01-26}
{\tilde p}_0
&=
\sum_{\substack{B_2\subset [m]\\
|B_2|=k}}
P_2(B_2)
\Biggl(
S_{3,k}
-
\sum_{\substack{B_3\subset [m],|B_3|=k\\
B_3\cap B_2\not=\emptyset}}
P_3(B_3)
\Biggr)
=
S_{2,k}S_{3,k}-{\bar R},
\end{align}
where 
${\bar R}
=
R_1+\dots+R_k$, 
and where
\begin{align*}
R_\ell
=
\sum_{\substack{B_2\subset [m]\\
|B_2|=k}}
P_2(B_2)
{\tilde S}_\ell(B_2),
\qquad
{\tilde S}_\ell(B_2)
=
\sum_{\substack{B_3\subset [m],|B_3|=k\\
|B_3\cap B_2|=\ell}}
P_3(B_3).
\end{align*}
Next, we evaluate the product $S_{2,k}S_{3,k}=S_{2,k}^2$ 
using (\ref{2026-01-12+3}) and show below that
${\bar R}=o
\left(
\left(\frac{m}{n}\right)^{2k-1}
\right)$. Now (\ref{2026-01-26}) implies 
 the first relation of 
(\ref{2026-01-10+5}).

It remains to show that ${\bar R}=o
\left(
\left(\frac{m}{n}\right)^{2k-1}
\right)$. Let us consider the sum
${\tilde S}_\ell(B_2)$.
Given $B_2$, we split $B_3=A\cup D$, 
where $A=B_3\cap B_2$  and $D\cap B_2=\emptyset$.
Clearly, $|A|=\ell$ 
and  $|D|=k-\ell$. 
Using (\ref{2026-01-10+6}) we upperbound
\begin{align*}
P_3(A)
=
\prod_{h\in A}p_h(2,1)
\le 
\prod_{h\in A}\frac{x_{a,h}}{n-1}
\le 
\prod_{h\in A}\frac{x_{h}}{n-1}
=
O
\left(
\left(
\frac{\sqrt{\ln n}}{\sqrt{n}}
\right)^\ell
\right)
\end{align*}
 uniformly in $A\subset [m]$, $|A|=\ell$. 
 It follows from the identity $P_3(B_3)=P_3(A)P_3(D)$
 that
\begin{align*}
{\tilde S}_\ell(B_2)
&
=
\sum_{\substack{A\subset B_2\\
|A|=\ell}}
P_3(A)
\sum_{\substack{D\subset [m]\setminus B_2\\
|D|=k-\ell}}
P_3(D)
\le
\sum_{\substack{A\subset B_2\\
|A|=\ell}}
P_3(A)
S_{3,k-\ell}
\\
&
\le 
O
\left(
\left(
\frac{\sqrt{\ln n}}{\sqrt{n}}
\right)^\ell
\right)
\binom{k}{\ell}
S_{3,k-\ell}.
\end{align*}
Since the bound hols uniformly over $B_2$, we have
\begin{align*}
R_\ell
=
O
\left(
\left(
\frac{\sqrt{\ln n}}{\sqrt{n}}
\right)^\ell
\right)
\binom{k}{\ell}
S_{3,k-\ell}
S_{2,k}
=
O
\left(
\left(
\frac{\sqrt{\ln n}}{\sqrt{n}}
\right)^\ell
\right)
O\left(\left(\frac{m}{n}\right)^{2k-\ell}
\right).
\end{align*}
In the last step we invoked the upper bounds for $S_{3,k-\ell}$ 
and
$S_{2,k}$ 
shown in (\ref{2026-01-12+2}).
It follows that ${\bar R}=R_1+\cdots+R_k=O
\left(
\frac{\sqrt{\ln n}}{\sqrt{n}}
\left(\frac{m}{n}\right)^{2k-1}
\right)$.
\end{proof}

\subsection{Inequalities related to expansion property}

 Recall that  ${\cal K}_n$ denote the complete graph on the 
 vertex 
set ${\cal V}=[n]$. 
 Given  graph $F=(V_F,E_F)$ (with $|V_F|\le n$ vertices), let $V^*_F$ be a subset of ${\cal V}$ selected uniformly at random from the family of 
 subsets of ${\cal V}$ of 
 size $|V_F|$.
 Given $V^*_F$ let $\pi:V^*_F\to V_F$ be a bijection selected uniformly at random.
 The subgraph $F^*=(V^*_F,E^*_F)$ of ${\cal K}_n$ where any 
 two vertices $x,y\in V^*_F$ are adjacent whenever  
$\pi(x), \pi(y)$ are adjacent in $F$ is called a random 
copy of $F$ (in ${\cal K}_n$).
We call the map $F\to F^*$ a random embedding.
We call a graph  $F$ basic if it is a union of independent 
($=$ non-incident) edges and/or paths of length $2$. Hence 
the minimal degree of a basic graph is one.
Paths of length $2$ are also called open triangles. A subgraph $F$ of a graph is called basic if it is a spanning subgraph and $F$ is basic.  We say that a subgraph of 
 ${\cal K}_n$ connects vertex sets $A, B\subset {\cal V}$,
 if it contains an edge with one endpoint in $A$ and the other one in $B$.

\begin{lem}\label{basic} Let $n, r, x$ be positive integers such that $2\le x\le n$  and $1\le r\le n/10$. 
Let $F$ be a graph on $x$ vertices having minimal degree
$\delta(F)\ge 1$. Let $F^*$ be a random copy of $F$ in 
${\cal K}_n$. We have
\begin{align}
\label{basic+}
\PP\left\{
F^* {\text{connects   sets\ }} [r] {\text{\ and \ }} [n]\setminus[r]
\right\}
&
\ge 
\frac{r}{n}
\left(
1-\frac{r}{n}
\right)
x
-
\frac{1}{2}\frac{r^2}{n^2}
x^2.
\end{align}
\end{lem}

\begin{proof}[Proof of Lemma \ref{basic}]
Denote for short $P_{r,n}(F)$ the probability on the left of (\ref{basic+}). Let $F_{\cal K}$ be the graph obtained from $F$ by replacing each component of $F$ by the clique having the same vertex set as the component. We observe  that
 $P_{r,n}(F)=P_{r,n}(F_{\cal K})$. Furthermore, for any 
basic subgraph   $F_B$  of $F_{\cal K}$ we have 
$P_{r,n}(F_{\cal K})\ge P_{r,n}(F_B)$.
Hence, $P_{r,n}(F)\ge P_{r,n}(F_B)$. Now (\ref{basic+}) follows from Lemma \ref{2025-11-20+8} below.
\end{proof}

 Note that inequality (\ref{basic+}) implies 
 \begin{align*}
 \PP\left\{
F^* {\text{connects   sets\ }} [r] {\text{\ and \ }} [n]\setminus[r]
\right\}
&
\ge 
\frac{r}{n}x
\left(
1
-
\min
\left\{
1,\frac{3}{2}\frac{r}{n}x
\right\}
\right).
\end{align*}
Indeed, for $\frac{3}{2}\frac{r}{n}x> 1$ the right side is negative and the inequality becomes trivial.
For $\frac{3}{2}\frac{r}{n}x\le 1$ the right side 
becomes $\frac{r}{n}x-\frac{3}{2}\frac{r^2}{n^2}x^2$. This quantity does not exceed the right side of (\ref{basic+}).

An immediate consequence of the latter inequality is the  following corollary  of Lemma \ref{basic}.
 
 \begin{col}\label{r_basic}
 Let $n, r$ be positive integers such that  
 $1\le r\le n/10$. 
Let $G$ be a random graph. Let $X$ denote the (possibly random) number of non-isolated vertices of $G$. We assume that the number of vertices of $G$ is at most
 $n$ 
with probability one. Let $G^*$ be a random copy of $G$ in ${\cal K}_n$.  We have
\begin{align}
\label{r_basic+}
\PP\left\{
G^* {\text{connects   sets\ }} [r] {\text{\ and \ }} [n]\setminus[r]
\right\}
\ge 
\frac{r}{n}
\E X
-
\frac{r}{n}
\E \left(
X
\min\left\{
1,\frac{3}{2}
\frac{r}{n}X
\right\}
\right).
\end{align}
 \end{col}
The random subgraph $G^*$ considered in Corollary \ref{r_basic} is the result of the two step procedure:
firstly we generate an instance $G=(V_G,E_G)$ of the random graph; secondly, given $G$, we generate its random copy $G^*$
(in ${\cal K}_n$). We assume that the two steps are stochastically independent. We note that  $G$  may have random number of edges and random configuration (an number) of edges. Also $G$ can be a deterministic graph,
when $\PP\{G=G_0\}=1$ for some deterministic graph $G_0$. 

We will apply Corollary \ref{r_basic} to random graphs
$G_{n,1},\dots, G_{n,m}$. Denote
\begin{align*}
q_{r,i}
=
1
-
\PP\{
G_{n,i} {\text{\ connects \ sets\ }} [r] {\text{ \ and \ }} [n]\setminus[r]\},
\qquad
i=1,\dots, m.
\end{align*} 
We will use  the short-hand notation
$\eta_r(x)=
x
\min\left\{
1,\frac{3}{2}
\frac{r}{n}x
\right\}$.
It follows from (\ref{r_basic+})  that
\begin{align}
\nonumber
q_{r,i}
&
\le
1
-
\frac{r}{n}
\E X_{n,i}
+
\frac{r}{n}\E
\left(
X_{n,i}
\min\left\{
1,\frac{3}{2}
\frac{r}{n}X_{n,i}
\right\}
\right)
\\
\label{q++}
&
=
1
-
\frac{r}{n}
\E X_{n,i}
+
\frac{r}{n}\E \eta_r(X_{n,i}).
\end{align} 

\begin{lem}\label{q-bounds}
Let $\beta\in(0,1)$.
Let $n\to+\infty$. Assume that $m=m(n)\to+\infty$.
Assume that 
(\ref{2025-11-24}) holds. 
The following statemens are true.

(i) For a sequence $\phi_n\downarrow 0$ satisfying (\ref{phi+})
we have for large $n$ that for any $1\le r\le n^{\beta}$
\begin{align}
\label{2025-11-27+6}
\max_
{H\subset [m]:\, |H|\ge m-\phi_n^{-1/4}}
\prod_{i\in H} q_{r,i}
\le 
e^{-r\frac{m}{n}\left(\kappa_n-\frac{2}{\ln n}\right)+r\phi_n^{1/4}}.
\end{align}

(ii) For sufficiently large $n$ we have for any 
$1\le r\le n^{\beta}$
\begin{align}\nonumber
\prod_{i\in [m]} q_{r,i}
\le 
e^{-r\frac{m}{n}\left(\kappa_n-\frac{2}{\ln n}\right)}.
\end{align}

(iii) We have for any 
$1\le r\le n/2$
\begin{align}
\nonumber
\prod_{i\in [m]} q_{r,i}
\le 
e^{-2\alpha_nr\frac{m}{n}\frac{n-r}{n-1}}.
\end{align}
\end{lem}

\begin{proof}[Proof of Lemma \ref{q-bounds}]
Given $H\subset [m]$ we introduce the sums 
\begin{align*}
S(H)=\sum_{i\in H}\E X_{n,i},
\qquad
R(H)
=
\sum_{i\in H}
\E
\eta(X_{n,i})
\end{align*}
and explore their properties. 
We observe that   
 $S_1([m])=m\kappa_n$ and $R([m])
=
m\E
\eta_r(X_{n,i_*})$. 
Next, we show that, given $0<\beta<1$, 
we have 
for large $n$ that 
$R([m])\le 2\frac{m}{\ln n}$, 
(equivalently,
$\E
\eta_r(X_{n,i_*})
\le 
\frac{2}{\ln n}$) uniformly in $1\le r\le n^{\beta}$.
We denote 
 $\tau=\frac{1}{\kappa_n\ln n}$ and split
\begin{align*}
\E \eta_r(X_{n,i_*})
&
=
\E 
\left(
\eta_r(X_{n,i_*})
\left(
{\mathbb I}_{\{\frac{r}{n}X_{n,i_*}\le \tau\}}
+
{\mathbb I}_{\{\frac{r}{n}X_{n,i_*}> \tau\}}
\right)
\right)
\\
&
\le 
\frac{3}{2}\frac{r}{n}\E \left(X_{n,i_*}^2
{\mathbb I}_{\{\frac{r}{n}X_{n,i_*}\le \tau\}}
\right)
+
\E \left(
X_{n,i_*}
{\mathbb I}_{\{\frac{r}{n}X_{n,i_*}> \tau\}}
\right)
=:I_1+I_2.
\end{align*}
Furthermore, we  estimate (recall the notation $\xi_n=X_{n,i_*}\ln(1+X_{n,i_*})$)
\begin{align}
\label{2025-11-29}
I_1
&
\le 
\frac{3}{2}
\tau \E X_{n,i_*}
=
\frac{3}{2}\frac{1}{\ln n},
\\
\label{2025-11-27-1}
I_2
&
\le
\frac{1}{\ln(1+\frac{n}{r}\tau)}
\E \left(
\xi_n
{\mathbb I}_{\{\frac{r}{n}X_{n,i_*}> \tau\}}
\right)
\\
\label{2025-11-27-2}
&
\le 
\frac{1}
{
\ln
\left(
1
+
n^{1-\beta}\tau
\right)
}
\E \left(
\xi_n
{\mathbb I}_{\{X_{n,i_*}> n^{1-\beta}\tau\}}
\right)
\\
\label{2025-11-27-3}
&
=
o\left(\frac{1}{\ln n}\right).
\end{align}
Inequality (\ref{2025-11-29}) together with bound (\ref{2025-11-27-3}) implies inequality $\E
\eta_r(X_{n,i_*})
\le 
\frac{2}{\ln n}$ for large $n$.
We comment on
steps (\ref{2025-11-27-1}--\ref{2025-11-27-3}). 
Inequality (\ref{2025-11-27-1}) follows from the fact 
that $X_{n,i_*}>\frac{n}{r}\tau$ 
implies $\ln (1+X_{n,i_*})>\ln (1+\frac{n}{r}\tau)$.
Inequality (\ref{2025-11-27-2}) follows from the 
inequality $r\le n^{\beta}$. Bound (\ref{2025-11-27-3})
follows from the  relation 
\begin{align*}
\ln
\left(
1
+
n^{1-\beta}\tau
\right)
\ge 
\ln
\left(
\frac{n^{1-\beta}}{\ln n}
\frac {1}{\kappa_n}
\right)
=
 (1-\beta)\ln n-\ln\ln n-\ln \kappa_n
 \sim 
 (1-\beta)\ln n
 \end{align*}
   and the fact that 
(\ref{2025-11-24})  implies $\E \left(
\xi_n
{\mathbb I}_{\{X_{n,i_*}> n^{1-\beta}\tau\}}
\right)=o(1)$ for $n^{1-\beta}\tau\to+\infty$.
At this step we also  use inequality
$\limsup_n\kappa_n<\infty$,  which 
follows from  (\ref{2025-11-24}).

Proof of statement  (i). 
(\ref{phi+}) implies 
$\max_{i\in[m]}\E X_{n,i}\le \tau_n^{-2}n$ for
 $\tau_n:=\phi^{-1/4}_n$. 
For any $H\subset [m]$ of size $|H|\ge m-\tau_n$ we have 
\begin{align*}
S(H)
=
S([m])-\sum_{i\in [m]\setminus H}\E X_{n,i}
\ge 
S([m])-\tau_n \max_{i\in [m]}\E X_{n,i}
\ge 
m\kappa_n-\frac{n}{\tau_n}.
\end{align*}
Now 
using (\ref{q++}) and inequality $1+t\le e^t$ we 
upper bound the product
\begin{align*}
\prod_{i\in H}q_{r,i}
\le e^{-\frac{r}{n}S(H)+\frac{r}{n}R(H)}
\le e^{-\frac{r}{n}S(H)+\frac{r}{n}R([m])}
\le 
e^{-r\frac{m}{n}\left(\kappa_n-\frac{2}{\ln n}\right)+\frac{r}{\tau_n}}.
\end{align*}

Proof of statement (ii). Proceeding as in the proof of (i) above
we estimate
\begin{align*}
\prod_{i\in [m]}
q_{r,i}
\le e^{-\frac{r}{n}S([m])+\frac{r}{n}R([m])}
= e^{-\frac{r}{n}m\kappa_n+\frac{r}{n}R([m])}
\le 
e^{-r\frac{m}{n}\left(\kappa_n-\frac{2}{\ln n}\right)}.
\end{align*}

Proof of statement (iii). Denote 
$\alpha_{n,i}=\PP\{G_{n,i}$ has at least one edge$\}$.  
Let $L^*$ be an edge of 
$G_{n,i}=({\cal V}_{n,i},{\cal E}_{n,i})$ selected uniformly at 
random from the set of edges ${\cal E}_{n,i}$ when this set is non-empty. Clearly, $L^*$ 
is a random edge of the complete graph ${\cal K}_n$. 
We have
\begin{align*}
q_{r,i}
&
\le 
1
-
\PP\bigl\{
L^* {\text{\, connects \, sets\, }} [r] {\text{ \ and \ }} [n]\setminus[r]\bigr|{\cal E}_{n,i}\not=\emptyset
\bigr\}
\PP\{{\cal E}_{n,i}\not=\emptyset\}
\\
&
=
1-2\frac{r(n-r)}{n(n-1)}\alpha_{n,i}.
\end{align*} 
Furthermore, $1+t\le e^t$ implies 
$q_{r,i}
\le 
e^{-2\alpha_{n,i}\frac{r}{n}\frac{n-r}{n-1}}$. 
Hence 
$\prod_{i\in [m]}q_{n,i}\le e^{-2\alpha_n r\frac{m}{n}\frac{n-r}{n-1}}$.
\end{proof}

In the remaining part  of the section we formulate and prove Lemma \ref{2025-11-20+8}, which may be of independent interest.

  Let $F_{k,\ell}=(V_{k,\ell}, E_{k,\ell})$ be basic graph, which is a union of nonincident
  edges $L_1,\dots, L_k$ and open triangles (cherries)  $T_1,\dots, T_{\ell}$. Hence 
  $|V_{k,\ell}|=2k+3\ell=:v_{k,\ell}$ and $|\E_{k,\ell}|=k+2\ell$.
  Let $F^*_{k,\ell}=(V_{k,\ell}^*,E_{k,\ell}^*)$ be a random copy of $F_{k,\ell}$. Note that $F^*_{k,\ell}$ is the union of nonincident random edges 
  $L^*_1,\dots, L^*_k$  and random open triangles $T^*_1,\dots, T^*_\ell$.
  Given positive integer $r<n$ introduce event
  \begin{align*}
  {\cal A}_{r,k,\ell}
  =
  \{F^*_{k,\ell} \ {\text{connects sets }} 
  [r]
  \
   {\text{and}}
   \  
   [n]\setminus [r]\}.
  \end{align*}
  
\begin{lem} \label{2025-11-20+8} For 
$r\le n/10$ we have
\begin{align}
\label{2025-09-24+6}
\PP\{{\cal A}_{r,k,\ell}\}
\ge 
 v_{k,\ell}\frac{r}{n}\left(1-\frac{r}{n}\right)
 -
 \frac{1}{2}
v^2_{k,\ell}\frac{r^2}{n^2}.
\end{align}
\end{lem}

\begin{proof}[Proof of Lemma \ref{2025-11-20+8}]
Let $r=1$. Let $v^*\in{\cal V}$  be a vertex selected uniformly at random  and independently of $F_{k,\ell}^*$. Relation (\ref{2025-09-24+6}) follows  from the identities
\begin{align*}
\PP\{{\cal A}_{1,k,\ell}\}=\PP\{1\in V_{k,\ell}^*\}
=\PP\{v^*\in V_{k,\ell}^*\}
=\frac{v_{k,\ell}}{n}.
\end{align*}

Let $r=2$. We evaluate the probability 
$1-\PP\{{\cal A}_{2,k,\ell}\}$ of the complement event 
${\bar {\cal A}}_{2,k,\ell}$. 
Let
$\{u^*,v^*\}\subset {\cal V}$  be a vertex pair selected uniformly at random and independently of $F_{k,\ell}^*$.
We have 
\begin{align*}
1-\PP\{{\cal A}_{2,k,\ell}\}
&
=
\PP\{\{u^*,v^*\}\in \{L^*_1,\dots, L^*_k\}\}
+
\PP\{\{u^*,v^*\}\cap V_{k,\ell}^*=\emptyset\}
\\
&
=
\frac{k}{\binom{n}{2}}
+
\frac{(n-v_{k,\ell})_2}{(n)_2}.
\end{align*}
Now, using $2k\le v_{k,\ell}$ one easily shows 
(\ref{2025-09-24+6}).

Let $r\ge 3$.
Introduce events 
\begin{align*}
{\cal L}_i
=
\{
L^*_i 
\,
{\text{connects sets }}
\,
[r]
\,
{\text{and}}
\
[n]\setminus[r]
\},
\quad
{\cal T}_j
=
\{
T^*_j
\,
{\text{connects sets }}
\,
[r]
\,
{\text{and}}
\
[n]\setminus[r]
\}.
\end{align*}
We write event ${\cal A}_{r,k,\ell}$ in the form 
${\cal A}_{r,k,\ell}
=
\left(\cup_{i}{\cal L}_i\right)\cup\left(\cup_{j}{\cal T}_j\right)$ and apply
the inclusion-exclusion inequalities
\begin{align}
\label{2025-09-26+2}
S_L+S_T-Q
\le
\PP\{{\cal A}_{r,k,\ell}\}
\le 
S_L+S_T,
\end{align}
where 
\begin{align*}
&
S_L
=
\sum_i\PP\{{\cal L}_i\}
=
k
\,
\PP\{{\cal L}_1\},
\quad 
S_T
=
\sum_{j}\PP\{{\cal T}_j\}
=
\ell
\,
\PP\{{\cal T}_{1}\},
\\
&
Q=Q_L+Q_T+Q_{LT},
\qquad
 Q_L
 =
 \sum_{i_1<i_2}\PP\{{\cal L}_{i_1}\cap{\cal L}_{j_2}\}
 =
 \binom{k}{2}
 \PP\{{\cal L}_{1}\cap{\cal L}_{2}\},
 \\
 &
 Q_T=\sum_{j_1<j_2}\PP\{{\cal T}_{i_1}\cap{\cal T}_{j_2}\}
 =
 \binom{\ell}{2}
 \PP\{{\cal T}_{1}\cap{\cal T}_{2}\},
 \quad
 Q_{LT}
 =
 \sum_{i}\sum_{j}
 \PP\{{\cal L}_{i}\cap{\cal T}_{j}\}
 =
 k\ell
 \,\PP\{{\cal L}_{1}\cap{\cal T}_{1}\}.
 \end{align*}

 We show below that
 \begin{align}
 \label{2025-09-26}
  &
  2\frac{r}{n}-2\frac{r^2}{n^2}
  <
 \PP\{{\cal L}_1\}
 <
 2\frac{r}{n},
 \\
 \label{2025-09-26+1}
 &
  3\frac{r}{n}-3\frac{r^2}{n^2}
 <
  \PP\{{\cal T}_1\}
  < 
 3\frac{r}{n},
 \\
 \label{2025-09-24+7}
  &
  \PP\{{\cal L}_1\cap {\cal L}_2\}
  <
  \PP\{{\cal L}_1\}
  \PP\{{\cal L}_2\},
  \qquad
  \
 {\text{for}}
 \quad
 r\le n/4,
 \\
 \label{2025-09-24+8}
 &  
 \PP\{{\cal L}_1\cap {\cal T}_1\}< \PP\{{\cal L}_1\}\PP\{{\cal T}_1\},
 \qquad
 \
 \
 {\text{for}}
 \quad
 r\le n/6,
  \\
  \label{2025-09-24+9}
  &
  \PP\{{\cal T}_2\cap {\cal T}_1\}<\PP\{{\cal T}_1\} \PP\{{\cal T}_2\},
  \qquad
  \
  \
  \
 {\text{for}}
 \quad
  r\le n/10.
  \end{align}
 
 Inequalities  (\ref{2025-09-26}),  (\ref{2025-09-26+1}) imply
 \begin{align}
\label{2025-11-20+9}
 v_{k,\ell}
 \frac{r}{n}
 -
 v_{k,\ell}\frac{r^2}{n^2}
 \le
 S_L+S_T
 \le
 v_{k,\ell}\frac{r}{n}.
 \end{align}
 Inequalities 
 (\ref{2025-09-24+7}),
 (\ref{2025-09-24+8}),
(\ref{2025-09-24+9}) imply for $r\le n/10$
\begin{align*}
Q
\le 
\binom{k}{2}
(\PP\{{\cal L}_1\})^2
+ 
\binom{\ell}{2}
(\PP\{{\cal T}_1\})^2
+
k\ell \PP\{{\cal L}_1\}\PP\{{\cal T}_1\}
\le 
\frac{1}{2}(S_L+S_T)^2
\le
\frac{1}{2} 
v_{k,\ell}^2\frac{r^2}{n^2}.
\end{align*}
In the last step we used second inequality of (\ref{2025-11-20+9}). Now (\ref{2025-09-24+6}) follows 
from (\ref{2025-09-26+2}),
\begin{align}
\nonumber
\PP\{{\cal A}_{r,k,\ell}\}
\ge 
S_L+S_T-Q
\ge
v_{k,\ell}\frac{r}{n}
\left(1-\frac{r}{n}\right)
-
\frac{1}{2}
v_{k,\ell}^2\frac{r^2}{n^2}.
\end{align}

It remains to show (\ref{2025-09-26}-\ref{2025-09-24+9}).

{\bf Proof of (\ref{2025-09-26}).}
 We  evaluate $\PP\{{\cal L}_1\}
 =
2\frac{r(n-r)}{n(n-1)}$ and invoke inequalities
 \begin{align*}
2\frac{r}{n} \ge 
2\frac{r(n-r)}{n(n-1)}
> 
2\frac{r(n-r)}{n^2}
 =
 2\frac{r}{n}-2\frac{r^2}{n^2},
 \end{align*}
{\bf Proof of  (\ref{2025-09-26+1}).}
We evaluate $\PP\{{\cal T}_1\}
 =
 1
 -
  \PP\{{\bar {\cal T}}_1\}
 =
 1
 -
 \frac{(r)_3}{(n)_3}
 -
 \frac{(n-r)_3}{(n)_3}$
 and invoke inequalities
  \begin{align*}
3\frac{r}{n}
\ge
 1
 -
 \frac{(n-r)_3}{(n)_3}
 >
 1
 -
 \frac{(r)_3}{(n)_3}
 -
 \frac{(n-r)_3}{(n)_3}
 > 
 1
 -
 \frac{r^3}{n^3}
 -
 \frac{(n-r)^3}{n^3}
 =
 3\frac{r}{n}-3\frac{r^2}{n^2}.
 \end{align*}

{\bf Proof of (\ref{2025-09-24+7}).}
We apply the product rule
$\PP\{{\cal L}_1\cap {\cal L}_2\}
 =
 \PP\{{\cal L}_1\}\PP\{{\cal L}_2|{\cal L}_1\}$
and estimate 
\begin{align*}
\PP\{{\cal L}_2|{\cal L}_1\}
= 
 \frac{(r-1)(n-r-1)}{\binom{n-2}{2}}
< 
\frac{r(n-r)}{\binom{n}{2}}
= 
\PP\{{\cal L}_1\}.
\end{align*}
The inequality above is equivalent to the inequality
\begin{align*}
\left(1-\frac{1}{r}\right)
\left(1-\frac{1}{n-r}\right)
<
\left(1-\frac{2}{n}\right)
\left(1-\frac{2}{n-1}\right),
\end{align*}
which follows from the inequality $1-\frac{1}{r}<\left(1-\frac{2}{n}\right)\left(1-\frac{2}{n-1}\right)$ valid for $r\le n/4$.

{\bf Proof of (\ref{2025-09-24+8}),
 (\ref{2025-09-24+9}).}
  Let ${\mathbb L}$ denote the set edges  of 
 ${\cal K}_n$
 connecting $[r]$ and $[n]\setminus[r]$.  Let ${\mathbb T}$ denote the set of open triangles $\tau$ of 
 ${\cal K}_n$
 connecting $[r]$ and $[n]\setminus[r]$. 
  Given a set $A$ we denote by ${\cal K}_A$ the 
 complete graph on the vertex set $A$.

{\bf Proof of (\ref{2025-09-24+8})}. 
We recall that $T_1^*$ is uniformly distributed across the 
set of open triangles of ${\cal K}_n$. Furthermore, given 
$T_1^*$, the random edge $L_1^*$ is uniformly distributed 
across the set of edges of ${\cal K}_n$ that are 
non-incident to $T_1^*$.
We partition 
 ${\mathbb T}
 =
 {\mathbb T}_1\cup{\mathbb T}_2
 \cup
 {\mathbb T}_3\cup{\mathbb T}_4$ so that

 $\tau\in {\mathbb T}_1\Leftrightarrow \tau$ contains an 
 edge in ${\cal K}_{[r]}$;
 
 $\tau\in {\mathbb T}_2\Leftrightarrow \tau$ contains an 
 edge in ${\cal K}_{[n]\setminus[r]}$;

 $\tau\in {\mathbb T}_3\Leftrightarrow \tau$ contains one vertex in 
 ${\cal K}_{[r]}$ and two non-adjacent vertices in ${\cal K}_{[n]\setminus[r]}$;
 
 $\tau\in {\mathbb T}_4\Leftrightarrow \tau$ contains one vertex in 
  ${\cal K}_{[n]\setminus[r]}$
   and two non-adjacent vertices in ${\cal K}_{[r]}$.

 We have, by the total probability formula,
 \begin{align*}
 \PP\{{\cal L}_1\cap {\cal T}_1\}
 &
 =
 \sum_{\tau\in{\mathbb T}}
 \PP\{{\cal L}_1|T^*_1=\tau\}
  \PP\{T^*_1=\tau\}
  =
\sum_{i=1}^4\sum_{\tau\in{\mathbb T}_i}
 \PP\{{\cal L}_1|T^*_1=\tau\}
  \PP\{T^*_1=\tau\}.
 \end{align*}
To prove (\ref{2025-09-24+8}) we show that 
$\PP\{{\cal L}_1|T^*_1=\tau\}
 \le \PP\{{\cal L}_1\}$
 for each $i=1,2,3,4$
 and every $\tau\in {\mathbb T}_i$.
 
Denote, for short, $N_L=|{\mathbb L}|=r(n-r)$ and $M_L=\binom{n}{2}$ so that   $\PP\{{\cal L}_1\}=\frac{N_L}{M_L}$. 
Similarly, for $\tau\in{\mathbb T}_i$ we denote
by $M_i$ (and $N_i$) 
the number of edges of ${\cal K}_n$ nonincident to $\tau$ (and connecting $[r]$ and $[n]\setminus[r]$) so that
$\PP\{{\cal L}_1|T^*_1=\tau\}
=\frac{N_i}{M_i}$.
Note that 
 $M_i=\binom{n-3}{2}$ for $i=1,2,3,4$. 
 We show that 
$\frac{N_i}{M_i}\le \frac{N_L}{M_L}$ for 
$i=1,2,3,4$ and $2\le r\le n/6$.

Let $i=1$. We have  
      $N_1=(n-r-1)(r-2)$. Inequalities 
\begin{align*}
\frac{N_1}{N_L}
&
=
\frac{(n-r-1)(r-2)}{(n-r)r}
=
\left(1-\frac{1}{n-r}\right)
\left(1-\frac{2}{r}\right)
< 
1-\frac{2}{r},
\\
\frac{M_1}{M_L}
&
=
\frac{(n-3)_2}{(n)_2}
=
\left(1-\frac{3}{n}\right)
\left(1-\frac{3}{n-1}\right)
=
1-\frac{6}{n}+\frac{6}{n(n-1)}
>
1-\frac{6}{n}
\end{align*}
imply 
$
\frac{N_1}{M_1}
<
\frac{N_L}{M_L}$.
Let $i=2$.  
We have   $N_2=(n-r-2)(r-1)$. Inequalities 
\begin{align*}
\frac{N_2}{N_L}
=
\frac{(n-r-2)(r-1)}{(n-r)r}
=
\left(1-\frac{2}{n-r}\right)
\left(1-\frac{1}{r}\right)
< 
1-\frac{1}{r}
\end{align*}
and 
$\frac{M_2}{M_L}
=
\frac{M_1}{M_L}
>
1-\frac{6}{n}
$
shown above imply
$\frac{N_2}{M_2}
<
\frac{N_L}{M_L}$.
The remaining cases $i=3,4$ are treated  in much the same 
way.

\smallskip
 
{\bf Proof of (\ref{2025-09-24+9}).}
 We have, by the total probability formula,
 \begin{align*}
 \PP\{{\cal T}_2\cap {\cal T}_1\}
 &
 =
 \sum_{\tau\in{\mathbb T}}
 \PP\{{\cal T}_2|T^*_1=\tau\}
  \PP\{T^*_1=\tau\}
  =
\sum_{i=1}^4
\sum_{\tau\in{\mathbb T}_i}
 \PP\{{\cal T}_2|T^*_1=\tau\}
  \PP\{T^*_1=\tau\}.
 \end{align*}
To prove (\ref{2025-09-24+9}) we show that 
$\PP\{{\cal T}_2|T^*_1=\tau\}
 \le \PP\{{\cal T}_2\}$
 for each $i=1,2,3,4$ and every $\tau\in {\mathbb T}_i$. In the proof we use the fact that given 
$T_1^*$, the random open triangle $T_2^*$ is uniformly distributed 
across the set of open triangles of ${\cal K}_n$ that are 
non-incident to $T_1^*$.
 
 Let  
 $N_T=|{\mathbb T}|=3\binom{r}{2}(n-r)+3\binom{n-r}{2}r$ denote
 the number of triangles 
 in $\mathbb T$ and 
$M_T=3\binom{n}{3}$ denote the number of open triangles of ${\cal K}_n$. 
We have $\PP\{{\cal T}_2\}=\frac{N_T}{M_T}$.
Similarly, for $\tau\in{\mathbb T}_i$ we denote
by 
$M_i$ (and $N_i$) the number of open triangles 
 of ${\cal K}_n$ non-incident to 
$\tau$ (and connecting $[r]$ and $[n]\setminus[r]$) so that
$\PP\{{\cal T}_2|T^*_1=\tau\}
=\frac{N_i}{M_i}$.
Note that 
 $M_i=3\binom{n-3}{3}$ for each $i=1,2,3,4$.
 We show that $\frac{N_i}{M_i}\le \frac{N_T}{M_T}$ for $i=1,2,3,4$.

Let $i=1$.
We evaluate the number  
$
N_1
=
3\binom{r-2}{2}(n-r-1)+ 3\binom{n-r-1}{2}(r-2)
$ and write it in the form
$
3\binom{r}{2}(n-r)\alpha+3\binom{n-r}{2}r \beta$,
where
\begin{align*}
\alpha
&
=\frac{(r-2)_2}{(r)_2}\frac{n-r-1}{n-r}
<
\frac{(r-2)_2}{(r)_2}
<
1-\frac{2}{r},
\\
\beta
&
=\frac{(n-r-1)_2}{(n-r)_2}\frac{r-2}{r}
<\frac{r-2}{r}
=
1-\frac{2}{r}.
\end{align*}
Hence
\begin{align*}
N_1
\le 
\max\{\alpha,\beta\}
\left(3\binom{r}{2}(n-r)+3\binom{n-r}{2}r \right)
<
\left(1-\frac{2}{r}\right)
N_T.
\end{align*}
Combining this inequality with the inequality
\begin{align}
\label{2025-09-24+5}
M_1
=
\frac{(n-3)_3}{(n)_3}
M_T
>
\left(1-\frac{10}{n}\right)
M_T,
\end{align}
which holds for  $n\ge 11$, we obtain for 
$n\ge 11$ and $r\le n/5$
\begin{align*}
\frac{N_1}{M_1}
< 
\frac{1-2r^{-1}}{1-10n^{-1}}\frac{N_T}{M_T}
\le 
\frac{N_T}{M_T}.
\end{align*}

Let $i=2$.  
We evaluate
$
N_2
=
3\binom{r-1}{2}(n-r-2)
+
3\binom{n-r-2}{2}(r-1)
$ 
and write it in the form
$
3\binom{r}{2}(n-r)\alpha+3\binom{n-r}{2}r \beta$,
where
\begin{align*}
\alpha
&
=\frac{(r-1)_2}{(r)_2}\frac{n-r-2}{n-r}
<
\frac{(r-1)_2}{(r)_2}
=
1-\frac{2}{r},
\\
\beta
&
=\frac{(n-r-2)_2}{(n-r)_2}\frac{r-1}{r}
<\frac{r-1}{r}
=
1-\frac{1}{r}.
\end{align*}
Hence
\begin{align*}
N_2
\le 
\max\{\alpha,\beta\}
\left(3\binom{r}{2}(n-r)+3\binom{n-r}{2}r \right)
<
\left(1-\frac{1}{r}\right)
N_T.
\end{align*}
This inequality combined with inequality
$\frac{M_2}{M_T}
=
\frac{M_1}{M_T}
>
1-\frac{10}{n}
$, see (\ref{2025-09-24+5})
above, 
implies 
for $n\ge 11$ and 
$r\le n/10$
\begin{align*}
\frac{N_2}{M_2}
< 
\frac{1-r^{-1}}{1-10n^{-1}}\frac{N_T}{M_T}
\le 
\frac{N_T}{M_T}.
\end{align*}
The remaining cases $i=3,4$ are treated  in much the same 
way. 
\end{proof}

\end{document}